\documentclass{amsart}

\usepackage[letterpaper,hmargin=1.5in,vmargin=0.8in]{geometry}

\usepackage{mathbbol}
\usepackage{thmtools}
\usepackage{thm-restate}
\usepackage{mathtools}
\usepackage{amsmath}
\usepackage{amsthm}
\usepackage{amssymb}
\usepackage{tikz}
\usetikzlibrary{matrix,arrows,cd,calc,positioning,arrows.meta}
\usepackage[hidelinks]{hyperref}
\usepackage{xcolor}
\definecolor{darkgreen}{rgb}{0,0.45,0}
\hypersetup{colorlinks,urlcolor=blue,citecolor=darkgreen,linkcolor=darkgreen,linktocpage}
\usepackage[capitalize]{cleveref}
\crefname{algocf}{Algorithm}{Algorithms}
\usepackage{xspace}
\usepackage{subcaption}
\usepackage[all,color]{xy}
\usepackage{enumitem}
\usepackage{comment}
\usepackage{booktabs}
\usepackage{multirow}
\usepackage[noend,ruled,linesnumbered]{algorithm2e}
\usepackage{algpseudocode}
\usepackage{afterpage}
\usepackage{placeins}

\makeatletter
\def\l@paragraph{\@tocline{4}{0pt}{1pc}{7pc}{}}
\def\l@subparagraph{\@tocline{5}{0pt}{1pc}{7pc}{}}
\makeatother

\makeatletter
\def\paragraph{\@startsection{paragraph}{4}%
  \z@\z@{-\fontdimen2\font}%
  {\normalfont\bfseries}}
\makeatother
\newtheorem{theorem}{Theorem}[section]
\newtheorem{lemma}[theorem]{Lemma}

\newtheorem{proposition}[theorem]{Proposition}

\newtheorem*{corollary*}{Corollary}

\newtheorem{theoremx}{Theorem}

\newtheorem{corollaryx}[theoremx]{Corollary}

\theoremstyle{definition}
\newtheorem{definition}[theorem]{Definition}
\newtheorem{example}[theorem]{Example}

\def\noteson{%
\gdef\luis##1{\noindent{\color{blue}[Luis: ##1]}}
\gdef\riju##1{\noindent{\color{violet}[Riju: ##1]}}
\gdef\thomas##1{\noindent{\color{brown}[Thomas: ##1]}}
\gdef\todo##1{\noindent{\color{red}[todo: ##1]}}}

\noteson

\renewcommand{\to}{\xrightarrow{\;\;\;}}

\renewcommand{\epsilon}{\varepsilon}
\renewcommand{\phi}{\varphi}

\DeclareMathOperator{\height}{\mathsf{height}}
\newcommand{\uf}{\mathsf{uf}}
\newcommand{\adduf}{\mathsf{add}}
\newcommand{\find}{\mathsf{find}}
\newcommand{\union}{\mathsf{union}}
\newcommand{\representatives}{\mathsf{representatives}}

\DeclareMathOperator{\End}{End}
\DeclareMathOperator{\add}{add}

\newcommand{\rep}{\mathbf{rep}}

\newcommand{\suc}{\mathsf{succ}}
\newcommand{\pred}{\mathsf{pred}}

\newcommand{\im}{\mathsf{im}}

\newcommand{\id}{\mathsf{id}}

\newcommand{\quiv}{\mathbf{quiv}}

\newcommand{\rtree}{\mathbf{rtree}}
\newcommand{\rtrees}{\mathbf{rtrees}}
\newcommand{\vect}{\mathbf{vec}}

\renewcommand{\top}{\mathbf{top}}
\newcommand{\free}{\mathsf{free}}
\newcommand{\disc}{\mathsf{disc}}

\newcommand{\set}{\mathbf{set}}

\newcommand{\merge}{\Gcal}
\newcommand{\fib}{\mathsf{fib}}

\DeclareMathAlphabet{\mathpzc}{OT1}{pzc}{m}{it}

\usepackage[mathscr]{euscript}

\makeatletter
\newcommand\DEFINEALPHABETLOOP[3]{%
  \ifx\relax#3\expandafter\@gobble\else\expandafter\@firstofone\fi
  {\expandafter\newcommand\expandafter*\csname#3#1\endcsname{#2{#3}}%
   \DEFINEALPHABETLOOP{#1}{#2}}%
}%
\newcommand\Definealphabet[2]{%
  \DEFINEALPHABETLOOP{#1}{#2}abcdefghijklmnopqrstuvwxyzABCDEFGHIJKLMNOPQRSTUVWXYZ\relax
}%
\makeatother
\Definealphabet{bb}{\mathbb}
\Definealphabet{cal}{\mathcal}
\Definealphabet{bf}{\mathbf}
\Definealphabet{sf}{\mathsf}
\Definealphabet{frak}{\mathfrak}
\Definealphabet{scr}{\mathscr}

\title[Decomposing zero-dimensional homology over rooted tree quivers]{Decomposing zero-dimensional persistent homology over rooted tree quivers}

\author{Riju Bindua}
\address{Indian Institute of Technology Delhi; New Delhi, India}
\author{Thomas Brüstle}
\address{Bishop's University; Université de Sherbrooke; Québec, Canada}
\author{Luis Scoccola}
\address{Centre de Recherches Mathématiques et Institut des sciences mathématiques;
Laboratoire de combinatoire et d'informatique mathématique de l'Université du Québec à Montréal;
Université de Sherbrooke; Québec, Canada}

\begin{document}

\maketitle

\begin{abstract}
Given a functor from any category into the category of topological spaces, one obtains a linear representation of the category by post-composing the given functor with a homology functor with field coefficients.
This construction is fundamental in persistence theory, where it is known as persistent homology, and where the category is typically a poset.
Persistence theory is particularly successful when the poset is a finite linearly ordered set, owing to the fact that in this case its category of representations is of finite type.
We show that when the poset is a rooted tree poset (a poset with a maximum and whose Hasse diagram is a tree) the additive closure of the category of representations obtainable as zero-dimensional persistent homology is of finite type, and give a quadratic-time algorithm for decomposition into indecomposables.
In doing this, we give an algebraic characterization of the additive closure in terms of Ringel's tree modules, and show that its indecomposable objects are the reduced representations of~Kinser.
\end{abstract}

\section{Introduction}

\subsection{Context}
Let $P$ be a rooted tree poset.
We study the category $\rep_{H_0}(P)$ of those linear representations of $P$ that are in the essential image of the zero-dimensional persistent homology functor $H_0 : \top^P \to \rep(P)$, where $\top$ is the category of topological spaces with finitely many path-connected components (equivalently, we study representations of $P$ that come from linearizing functors $P \to \set$).

We now give context and briefly discuss our results; we detail our contributions in the next section.
For some of the main notions involved in this paper, see \cref{table:notions}.

\subsubsection*{Poset representations in applied topology}
Geometric data can be studied using tools from algebraic topology, notably homology \cite{ghrist}.
Given a geometric dataset, various constructions from topological data analysis can be used to build a nested family of topological spaces $K_1 \subseteq \cdots \subseteq K_n$, called a \emph{filtration}, which encodes the topology of the dataset.
By applying homology with field coefficients to this sequence, one obtains a representation of a quiver of type $A_n$, called the \emph{persistent homology} of the filtration.
This representation can be effectively decomposed into indecomposable summands, thanks in part to the fact that $A_n$ has finite representation type.
Each of the indecomposable summands can be interpreted as a topological feature of the data, enabling the incorporation of topological information in statistical and machine learning applications \cite{hensel-moor-rieck,chazal-michel}.
More recently, this collection of ideas, now known as persistence theory \cite{oudot}, has also found theoretical applications in geometry and analysis \cite{shelukhin,polterovich,buhovsky}.

If $P$ is any partially ordered set, a \emph{$P$-filtration} is a collection of topological spaces $\{K_p\}_{p \in P}$ such that $K_p \subseteq K_q$ whenever $p \leq q \in P$.
For example, the filtrations in the previous paragraph are $P$-filtrations for $P = \{1 < \dots < n\}$.
The persistent homology of a $P$-filtration is a representation of the indexing poset $P$ (that is a functor $P \to \vect$), and we denote the category of such representations by $\rep(P)$.

Applications of topology to time-dependent, noisy, or otherwise parametrized data, motivate developing a persistence theory for filtrations indexed by non-linear posets \cite{botnan-lesnick-2}.
The first issue encountered is that non-linear posets are typically not of finite representation type \cite{loupias}, and usually of wild representation type.
This means that a classification of the indecomposable objects of $\rep(P)$ is generally infeasible, which in turn implies that
directly using the indecomposable decomposition of a certain representation as a statistic is hard or impossible.

\subsubsection*{Zero-dimensional persistent homology}
In several applications, notably clustering \cite{chazal-guibas-oudot-skraba,cai-kim-memoli-wang,rolle-scoccola}, the most relevant homological degree is zero.
Since the zero-dimensional homology of a topological space is conceptually and computationally simpler than higher dimensional homology, it seems plausible that the category $\rep_{H_0}(P)$ of representations of $P$ that can be obtained as the zero-dimensional persistent homology $H_0(X)$ of a $P$-filtration $X$ is simpler than the full category of representations $\rep(P)$.
Note that, for any higher homological degree $i > 0$, it is known that the analogous category $\rep_{H_i}(P)$ is equivalent to the whole category of representations $\rep(P)$ (see \cite[Section~4]{BBP}), when $P$ is a finite poset and homology has prime-field coefficients.

With this motivation, the authors of \cite{bauer-et-al} consider the image of zero-dimensional persistent homology for $P = \{1 < \cdots < m\} \times \{1 < \cdots < n\}$, a Cartesian product of two finite linear orders.
Since common filtrations in topological data analysis have the property that the horizontal morphisms of their zero-dimensional persistent homology are epimorphisms,
they focus on the subcategory $\rep^{e,\ast}(P) \subseteq \rep(P)$ of representations with the property that the horizontal structure morphisms are epimorphic.
Their main conclusion \cite[Corollary~1.6]{bauer-et-al} says that, except when $P$ belongs to a finite set of small grids, the category $\rep^{e,\ast}(P)$ is still of wild type.
Nevertheless, they observe that $\rep^{e,\ast}(P) \not\subseteq \rep_{H_0}(P)$ in general, and ask whether more can be said about representations in the image of $H_0$.
In this work, we conduct such a study in the case where $P$ is a rooted tree poset (see \cref{figure:rooted-tree,section:tree-quivers-tree-posets}).

We show in particular that for a rooted tree poset $P$, the category $\add(\rep_{H_0}(P))$ is of finite type (\cref{corollary:finite-representation-type}).
This is interesting for two reasons: even for a rooted tree poset~$P$, the category $\rep(P)$ is usually of wild type (e.g., when $P$ is a six-element poset with five incomparable elements all smaller than the sixth one), and $\add(\rep_{H_0}(P))$ can be of infinite type when $P$ is the opposite of a rooted tree poset (e.g., when $P$ is a five-element poset, with four incomparable elements all larger than the fifth element \cite[Example~5.13]{BBP}).

\begin{figure}
    \[
        \begin{tikzpicture}
            \begin{scope}[yshift=-2cm]
                \node (A1) at (1,0) {$\bullet$};
                \node (B1) at (0,0) {$\bullet$};
                \node (B2) at (-1,0) {$\bullet$};
                \node (C1) at (0.5,1) {$\bullet$};
                \node (C2) at (-0.5,1) {$\bullet$};
                \node (D1) at (0,2) {$\bullet$};

                \draw[-] (A1)--(C1);
                \draw[-] (B1)--(C1);
                \draw[-] (B2)--(C2);
                \draw[-] (C1)--(D1);
                \draw[-] (C2)--(D1);
            \end{scope}
        \end{tikzpicture}
        \hspace{3cm}
        \begin{tikzpicture}
            \begin{scope}[yshift=-2cm]
                \node (A1) at (0,1) {$\bullet$};
                \node (B1) at (0,0) {$\bullet$};
                \node (B2) at (0,-1) {$\bullet$};
                \node (C1) at (1,0.5) {$\bullet$};
                \node (C2) at (1,-0.5) {$\bullet$};
                \node (D1) at (2,0) {$\bullet$};

                \draw[->] (A1)--(C1);
                \draw[->] (B1)--(C2);
                \draw[->] (B2)--(C2);
                \draw[->] (C1)--(D1);
                \draw[->] (C2)--(D1);
            \end{scope}
        \end{tikzpicture}
    \]
    \caption{
        A rooted tree can be seen both as a poset and as a quiver.
        \emph{Left.}~The Hasse diagram of a rooted tree poset $P$.
        \emph{Right.}~The corresponding rooted tree quiver $Q_P$.
        The categories of representations are equivalent $\rep(P) \simeq \rep(Q_P)$
        (\cref{section:poset-quiver-representations}).
    }
    \label{figure:rooted-tree}
\end{figure}

\subsubsection*{Ringel's tree modules and Kinser's reduced representations}
Representations of a rooted tree poset can equivalently be seen as representations of a \emph{rooted tree quiver} $Q$ (see \cref{figure:rooted-tree,section:poset-quiver-representations}).
In this paper, we take the point of view of quivers rather than that of posets, since this has some technical advantages, and since it connects more readily with work of Ringel \cite{Ri} and of Kinser \cite{K10}, which we now present.

A representation of a quiver $Q$ is a \emph{tree module} \cite{Ri} if there exists a basis for the representation such that the coefficient quiver of the representation with respect to this basis (\cref{definition:coefficient-quiver}) is a tree quiver.
This implies, for example, that tree modules admit a basis for which all of their structure morphisms have only $0$ and $1$ in their matrix representation.

For $Q$ a rooted tree quiver, \cref{Theorem A}(2) says that representations in $\rep_{H_0}(Q)$ are direct sums of tree modules, and in fact of \emph{rooted tree modules} (i.e., representations admitting a basis for which the coefficient quiver is a rooted tree quiver).

In order to prove that $\add(\rep_{H_0}(Q))$ is of finite type, we show in \cref{Theorem B} that any rooted tree module over a rooted tree quiver decomposes as a direct sum of \emph{reduced rooted tree modules} (\cref{definition:reduced-rooted-tree-module}).
These objects originate in Kinser's work \cite{K10} on the representation ring of rooted tree quivers.
Rather than starting from Kinser's inductive definition, we use the following intrinsic characterization as our guiding notion:
 a rooted tree quiver $T$ over $Q$ is \emph{reduced} precisely when it has no non-identity endomorphisms over $Q$.
This is our \cref{definition:reduced}; later, \cref{proposition:characterization-reduced-trees} shows that it is equivalent to Kinser's recursive condition involving a certain preorder $\preceq_Q$.

\subsubsection*{The elder rule for rooted tree modules over rooted tree quivers}
The decomposition algorithms we introduce in \cref{thm::algorithm} make use of a generalization of the \emph{elder rule}, a concept from persistence theory \cite[pp.~188]{edelsbrunner-harer}, key in efficiently decomposing zero-dimensional persistent homology over linearly ordered sets, and revisited in different contexts such as \cite[Theorem~3.10]{curry}, \cite[Theorem~4.4]{cai-kim-memoli-wang}, and \cite[Lemma~117]{rolle-scoccola}.

The version of the elder rule relevant in our setup is our \cref{proposition:split-off-summand}, which makes use
of a partial preorder on the collection of rooted tree modules over a rooted tree quiver (\cref{definition:order-on-trees}), which is a slight generalization of the  order introduced by Kinser \cite{K10} for different purposes.
See also \cite[Lemma~1]{KM} for another, related, algebraic incarnation of the elder rule.

\subsection{Contributions}

Recall that any finite quiver $Q$ has an associated path category (\cref{definition:path-category}), which we also denote by $Q$,
and that there is a canonical equivalence of categories between $\rep(Q)$, the category of finite dimensional representations of $Q$ over some field $\kbb$, and $\vect^Q$, the category of functors from the path category of $Q$ to the category of finite dimensional $\kbb$-vector spaces.
Post-composition with $H_0 : \top \to \vect$ gives the \emph{zero-dimensional persistent homology} functor $H_0 : \top^Q \to \rep(Q)$.
Let $\rep_{H_0}(Q) \subseteq \rep(Q)$ denote the essential image of the zero-dimensional persistent homology functor, that is the collection of representations of $Q$ that can be obtained, up to isomorphism, as the zero-dimensional persistent homology of a functor $Q \to \top$.

A \emph{rooted tree quiver} $(T,f_T)$ \emph{over} a rooted tree quiver $Q$ consists of a rooted tree quiver $T$ together with a root-preserving quiver morphism $f_T: T \to Q$.
One of the contributions of this paper is to give an inductive description of rooted tree quivers and of rooted tree quivers over a rooted tree quiver (\cref{section:inductive-rooted-tree-quivers}).
Any such rooted tree quiver over $Q$ can be \emph{linearized} by pushing forward the constant representation of $T$ along $f_T$ (\cref{definition:linearized-rooted-tree-quiver}) to obtain a representation $\kbb_T \in \rep(Q)$.

Our first main result gives an algebraic characterization of the essential image of the zero-dimensional persistent homology functor over rooted tree quivers, and of the additive closure of the essential image (recall that the additive closure of a collection of objects is formed by taking all direct sums of all direct summands of the objects in the collection).

\begin{theoremx}\label{Theorem A}
    Let $Q$ be a rooted tree quiver. Then
    \begin{align}
        \rep_{H_0}(Q)       & = \Big\{\, \text{finite direct sums of linearized rooted tree quivers over $Q$}\, \Big\} \label{equation:image-h0} \\
        \add(\rep_{H_0}(Q)) & = \Big\{\, \text{finite direct sums of rooted tree modules over $Q$} \,\Big\}
        \label{equation:add-image-h0}
    \end{align}
    where the notation on the right-hand side of the equations denotes the corresponding subcategory of $\rep(Q)$ spanned by objects isomorphic to those specified. 
\end{theoremx}

The second main result, together with Kinser's work \cite{K10}, implies that, when $Q$ is a rooted tree quiver, the category $\add(\rep_{H_0}(Q))$ is of finite type, and the indecomposables are the reduced rooted tree modules.

Reduced rooted tree modules were introduced in~\cite{K10} under the name of ``reduced representations'', and they are linearizations of reduced rooted trees over (subtrees of) $Q$.
For readability, we take as our working definition the  condition
that a reduced rooted tree is one having no non-trivial endomorphism over $Q$.
The recursive formulation appearing in Kinser's definition is recalled later, after the preorder $\preceq_Q$ has been related to morphisms of rooted trees; see \cref{proposition:characterization-reduced-trees}.

\begin{theoremx}\label{Theorem B}
    Let $Q$ be a rooted tree quiver.
    Every rooted tree module over $Q$ decomposes as a direct sum of reduced rooted tree modules over $Q$.
\end{theoremx}

\begin{corollaryx}
    \label{corollary:finite-representation-type}
    Let $Q$ be a rooted tree quiver.
    The category $\add(\rep_{H_0}(Q))$ is of finite type, and the indecomposables are the reduced rooted tree modules over $Q$.
\end{corollaryx}

The third main contribution consists of algorithms for the effective decomposition of rooted tree modules over rooted tree quivers.

\begin{theoremx}\label{thm::algorithm}
    Let $Q$ be a rooted tree quiver.
    \begin{enumerate}
        \item Given $T$, a rooted tree quiver over $Q$, \cref{Algorithm 1} returns the summands in the indecomposable decomposition of $\kbb_T \in \rep(Q)$ in $O(|T|^2)$ time.
        \item Given $(G,f)$ a $Q$-filtered graph, \cref{Algorithm 2} returns the summands in the indecomposable decomposition of $H_0(f) \in \rep(Q)$ in $O(|G|^2)$ time.
    \end{enumerate}

\end{theoremx}

We conclude the paper by describing in \cref{section:TDA-applications} how our results can be used to study morphisms between merge trees, as well as the zero-dimensional persistent homology of two-parameter filtrations.

\begin{table}
    \centering
    {\footnotesize
        \hspace*{-0.07cm}%
        \begin{tabular}{ |l |l |  }
            \hline
            Rooted tree poset (a type of poset)                                                          & Sec.~\ref{section:inductive-rooted-tree-quivers} \\
            Rooted tree quiver (a type of quiver)                                                        & Sec.~\ref{section:background-quiver}             \\
            Rooted tree quiver over a rooted tree quiver $Q$ (a type of quiver morphism $T \to Q$)                     & Sec.~\ref{section:categories-of-quivers}         \\
            Linearized rooted tree quiver over a rooted tree quiver $Q$ (a type of representation of $Q$)  & Def.~\ref{definition:linearized-rooted-tree-quiver}      \\
            Rooted tree module over a quiver $Q$ (a type of representation of $Q$)                                    & Def.~\ref{definition:rooted-tree-module}             \\
            Reduced rooted tree quiver over a rooted tree quiver $Q$ (a type of quiver morphism $T \to Q$) & Def.~\ref{definition:reduced}                    \\
            Reduced rooted tree module over a rooted tree quiver $Q$ (a type of representation of $Q$)                & Def.~\ref{definition:reduced-rooted-tree-module} \\
            Order $\leq_Q$ (an order on vertices of a tree quiver $Q$)                                      & Def.~\ref{definition:order-on-quiver}            \\
            Preorder $\preceq_Q$ (a preorder on rooted tree quivers over a rooted tree quiver $Q$)             & Def.~\ref{definition:order-on-trees}             \\
            \hline
        \end{tabular}
    }
    \caption{Table of main notions.}
    \label{table:notions}
\end{table}

\subsection{Other related work in persistence}
In order to describe connections to the literature, we use the language of this paper freely.

In \cite{alonso-kerber}, idempotent endomorphisms on sets are studied, and it is shown that they induce direct sum decompositions when applying the linearization functor.
We relate this result to our approach in the proof of \cref{proposition:split-off-summand}.

In \cite{BBP}, a class of representations whose linear maps are given only by the numbers $0$ and $1$ is studied.
This class is related, but distinct from that of tree modules: A representation $V$ of $Q$ is called a \emph{component module} if it admits a collection of bases $B=\bigsqcup_{q \in Q_0}B_q$ in which the matrices $V_\alpha$ for each arrow $\alpha$ contain only ones and zeros, with exactly one $1$ in each column.
If the matrices contain at most one 1 in each column, then $V$ is called a \emph{semi-component module}.
Since every non-root vertex in a rooted tree quiver has exactly one direct successor, every rooted tree module $\kbb_T $ is in fact a component module.
Conversely, let  $V$ be a semi-component module whose coefficient quiver $\Gamma(V,B)$ is a rooted tree quiver.
Then $V\cong \kbb_T$ is a rooted tree module over $Q$ with $T=\Gamma(V,B)$ and where the graph map $f: \Gamma(V,B) \to Q$ is given by $f(a) = q$ if $a \in B_q.$

Rooted tree posets are particular instances of posets of dimension one, in the sense of \cite{chacholski-giunti-landi-tombari}.

\section{Background}

Throughout this article we fix a field $\kbb$, all vector spaces are taken to be $\kbb$-vector spaces, and we let $\vect$ denote the category of finite dimensional $\kbb$-vector spaces.
All quivers are assumed to be finite, and all representations are assumed to be finite dimensional.
We assume familiarity with basic category theory.

\subsection{Quivers, tree quivers, and rooted tree quivers}
\label{section:background-quiver}
A \emph{(finite) quiver} $Q=(Q_0,Q_1,s,t)$ consists of a finite \emph{vertex set} $Q_0$ and a finite \emph{edge set} $Q_1$, together with two functions $s,t : Q_1 \to Q_0$, called \emph{source map} and \emph{target map} respectively.
Note that a quiver is equivalently a directed graph possibly with loops and multi-edges.

A quiver is a \emph{tree quiver} if its underlying graph, when seen as undirected graph, is a tree (i.e., it is connected and without cycles).
A \emph{rooted tree quiver} is a tree quiver $Q$ with exactly one vertex $\sigma \in Q_0$ with out-degree zero.
The vertex $\sigma$ is called the \emph{root} of $Q$.
We sometimes denote a rooted tree $Q$ with root $\sigma$ by $(Q,\sigma)$.

The one-vertex and no-edge quiver is a rooted tree quiver, which we denote by $\ast$.

\subsection{Tree quivers and tree posets}
\label{section:tree-quivers-tree-posets}
The proofs in this subsection are straightforward exercises, and are hence omitted.

\smallskip

Reachability through a directed path defines, for every tree quiver $Q$, a partial order $\leq_Q$ on $Q_0$, which we now describe\footnote{For readers familiar with \cite{K10}, we mention that this order is the opposite of the one used in \cite{K10}.}.

\begin{definition}
    \label{definition:order-on-quiver}
    Let $Q$ be a quiver and let $x,y \in Q_0$.
    A \emph{directed path} from $x$ to $y$ in $Q$ is a (possibly empty) list $L$ of edges of $Q$ such that, if $L$ is empty, then $x=y$, and if $L = \{\alpha_1, \dots, \alpha_k\}$, then $s(\alpha_1) = x$, $t(\alpha_k) = y$, and $t(\alpha_i) = s(\alpha_{i+1})$ for all $1 \leq i \leq k-1$.
    If $Q$ is a tree quiver, let $x \leq_Q y$ if and only if there exists a directed path from $x$ to $y$.
\end{definition}

All principal downsets of a rooted tree quiver form themselves a rooted tree quiver, in the following sense.

\begin{definition}
    If $Q$ is a tree quiver (or, more generally, a disjoint union of tree quivers), and $x \in Q_0$, let $Q_{\leq x}$ denote the subquiver of $Q$ spanned by vertices $y \in Q_0$ such that $y \leq_Q x$.
    The quiver $Q_{\leq x}$ is a rooted tree quiver, with root $x$.
\end{definition}

\begin{lemma}
    \label{lemma:rooted-tree-lattice}
    A tree quiver $Q$ is a rooted tree quiver precisely when the poset $(Q_0, \leq_Q)$ has a unique maximal element (which is necessarily equal to the root of $Q$).
    Moreover, if $Q$ is a rooted tree quiver, the poset $(Q_0, \leq_Q)$ is a join-semilattice (a.k.a.~upper semilattice), that is for all $x, y \in Q_0$, the \emph{join} $x \vee y$ (i.e., greatest lower bound) of $\{x,y\}$ exists (and is thus unique).
    \qed
\end{lemma}

\begin{lemma}
    If $(Q,\sigma)$ is a rooted tree quiver and $x \neq \sigma \in Q_0$, then $x$ has a unique cover in the order $\leq_Q$.
    \qed
\end{lemma}

\begin{definition}
    If $(Q,\sigma)$ is a rooted tree quiver and $x \neq \sigma \in Q_0$,
    we denote the unique cover of $x$ in the order $\leq_Q$ by $\suc(x)$ and refer to it as the \emph{successor} of $x$.
    The set of predecessors of $x$ is $\pred(x) = \{y \in Q_0 \setminus \{\sigma\} : \suc(y) = x\}$.
\end{definition}

\begin{definition}
A \emph{tree poset} is a poset whose Hasse diagram is a tree.
A \emph{rooted tree poset} is a tree poset with a maximum.
\end{definition}

Given a tree poset $P$ we construct a quiver $Q_P$ as follows.
As vertices we put the elements of the underlying set of $P$, and for every pair $x,y \in P$ such that $y$ covers $x$, we put an edge $x \to y$.

\begin{lemma}
    If $Q$ is a (rooted) tree quiver, then $(Q_0, \leq_Q)$ is a (rooted) tree poset.
    If $P$ is a (rooted) tree poset, then $Q_P$ is a (rooted) tree quiver.
    \qed
\end{lemma}

\subsection{Quiver representations}
Let $Q = (Q_0, Q_1, s, t)$ be a quiver.
A (finite dimensional) \emph{representation} $V$ of $Q$ consists of a vector space $V_x$ for each vertex $x \in Q_0$, and a linear morphism $V_{\alpha} : V_{s(\alpha)} \to V_{t(\alpha)}$ for each edge $\alpha \in Q_1$.
If $V$ and $W$ are representations of $Q$, a \emph{morphism} $f : V \to W$ consists of a morphism $f_x : V_x \to W_x$ for every $x \in Q_0$, with the property that $W_{\alpha} \circ f_{s(\alpha)} = f_{t(\alpha)} \circ V_{\alpha}$.
Representations over $Q$ together with morphisms between them form the \emph{category of representations} of $Q$ denoted $\rep(Q)$.

The \emph{zero representation}, denoted $\mathbb{0}$, is the (unique) representation with $\mathbb{0}_x = 0$ for every $x \in Q_0$.
The \emph{constant representation}, denoted $\mathbb{1}_Q \in \rep(Q)$, is the representation such that $(\mathbb{1}_Q)_x = \kbb$ for every $x \in Q_0$, and $(\mathbb{1}_Q)_\alpha$ is the identity map $\kbb \to \kbb$ for every $\alpha \in Q_1$.

The category $\rep(Q)$ is an additive category \cite[Section 1.2]{DW}.
As in any additive category, a representation is said to be \emph{indecomposable} if it is not the zero representation, and it is also not isomorphic to the direct sum of two non-zero representations.
The category $\rep(Q)$ is Krull--Schmidt \cite[Theorem 1.7.4]{DW}, in the sense that
every representation $V \in \rep(Q)$ is a finite direct sum of indecomposable representations, and these indecomposables are unique up to isomorphism.

The next result allows us to switch from posets to quivers when studying representations of rooted tree posets.

\begin{lemma}
    \label{section:poset-quiver-representations}
    If $P$ is a tree poset, the category of functors $P \to \vect$ is equivalent to the category $\rep(Q_P)$.
\end{lemma}
\begin{proof}
    The proof is a routine check, so let us only outline the steps involved.
    Given a functor $F : P \to \vect$, we define a representation of $V$ of $Q_P$ as follows.
    We let $V_x = F(x)$ for every $x \in P$.
    Now, by construction, each edge $(x,y)$ in $Q_P$ corresponds to a pair $x,y \in P$ such that $y$ covers $x$, so we define $V_{(x,y)} = F(x \leq y) : V_x \to V_y$.
    This induces a functor $\vect^P \to \rep(Q_P)$.

    Similarly, given a representation $V \in \rep(Q_P)$, we define a functor $F : P \to \vect$ as follows.
    We let $F(x) = V_x$ for every $x \in P$.
    If $x \lneq y \in P$, let $x_1, \dots, x_k \in P$ be such that $x_1 = x$, $x_k = y$, and $x_{i+1}$ covers $x_i$.
    Then, define $F(x \leq y) = V_{(x_{k-1},x_{k})} \circ \cdots \circ V_{(x_1, x_2)} : F(x) \to F(y)$.
    To check that $F$ is indeed a functor, one uses the fact that $Q_P$ is a tree quiver, so that there exists a unique directed path from $x$ to $y$ in $Q_P$.
    This induces a functor $\rep(Q_P) \to \vect^P$.

    Finally, one checks that the two functors we have defined are inverse equivalences of categories.
\end{proof}

\subsection{Categories of quivers}
\label{section:categories-of-quivers}
A \emph{quiver morphism} $f : Q \to Q'$ from a quiver $Q=(Q_0, Q_1, s, t)$ to a quiver $Q' = (Q_0', Q_1', s', t')$ consists of functions $f_0 : Q_0 \to Q'_0$ and $f_1 : Q_1 \to Q'_1$ that respect source and target, in the sense that that $s' \circ f_1 = f_0 \circ s$ and $t' \circ f_1 = f_0 \circ t$.
Quivers and quiver morphisms form a category that we denote by $\quiv$.

A \emph{rooted tree quiver morphism} is a quiver morphism $f : Q \to Q'$, where $(Q,\sigma)$ and $(Q',\sigma')$ are rooted tree quivers, such that $f_0(\sigma) = \sigma'$.
Rooted tree quivers and rooted tree quiver morphisms form a (not full) subcategory $\rtree \subseteq \quiv$.
Note that the one-vertex rooted tree quiver $\ast$ is initial in $\rtree$, in the sense that there exists a unique rooted tree quiver morphism $u_\ast : \ast \to Q$ for every rooted tree quiver $Q$.

If $Q$ is a rooted tree quiver, we can consider the slice category $\rtree_{/Q}$, which has as objects the rooted tree quiver morphisms $f : T\to Q$, and as morphisms from $f : T \to Q$ to $f' : T' \to Q$ the set of morphisms $g : T \to T'$ such that $f \circ g = f'$.
We refer to any object of $\rtree_{/Q}$ as a \emph{rooted tree quiver over} $Q$.
Then, $\rtree_{/Q}$ is naturally a full subcategory of the slice category $\quiv_{/Q}$.

By a standard abuse of notation, we sometimes denote an object $(T,f_T : T \to Q)$ of a slice category simply as $T$.

\subsection{Inductive description of rooted tree quivers}
\label{section:inductive-rooted-tree-quivers}
We start by giving an inductive definition of rooted tree quivers.

\begin{figure}
    \includegraphics[width=\linewidth]{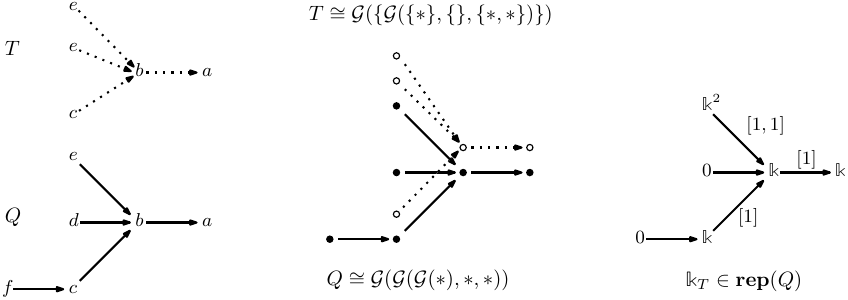}
    \caption{
        A rooted tree quiver $T$ over a rooted tree quiver $Q$ and its corresponding representation $\kbb_T \in \rep(Q)$.
        \emph{Left.}
        An illustration of $T \to Q$ given by labeling the vertices of $Q$ with distinct letters, and labeling the vertices of $T$ with the label of their image.
        \emph{Center.}
        Another illustration of $T \to Q$, as well as its construction as an inductive rooted tree quiver over an inductive rooted tree quiver.
        \emph{Right.}
        The linearization $\kbb_T \in \rep(Q)$.}
    \label{figure:inductive-tree-definitions}
\end{figure}

\begin{definition}
    \label{definition:inductive-rooted-tree-quiver}
    An \emph{inductive rooted tree quiver} is:
    \begin{itemize}[leftmargin=2cm]
        \item[(Base case)] either the quiver $\ast$ with one vertex and no edges;
        \item[(Ind.~\hspace{0.12cm}case)] or a quiver of the form $\merge(Q_1, \dots, Q_k)$, where $Q_1, \dots, Q_k$ are inductive rooted tree quivers, and where $\merge(Q_1, \dots, Q_k)$ is constructed by first taking the disjoint union of $Q_1, \dots, Q_k$, then adjoining a new vertex $\sigma$, and  adding a single edge from the root of $Q_i$ to $\sigma$, for each $1 \leq i \leq k$.
    \end{itemize}
\end{definition}

See \cref{figure:inductive-tree-definitions} for an illustration.
We refer to the operation $\merge$ as the \emph{gluing} operation.

\begin{lemma}
    \label{lemma:inductive-construction-rooted-trees}
    Every object of $\rtree$ is isomorphic to an inductive rooted tree quiver.
\end{lemma}
\begin{proof}
    Let $Q \in \rtree$.
    This is proven by induction on the number of vertices in $Q$.
    If $Q$ has a unique vertex, then it is isomorphic to the inductive rooted tree quiver with a single vertex.
    If $Q$ has more than one vertex, by removing the root and all adjacent edges from $Q$, one obtains a disjoint union of rooted tree quivers $Q_1, \dots, Q_k$, which by inductive hypothesis, must be isomorphic to inductive rooted tree quivers $Q'_1, \dots, Q'_k$, respectively.
    Then, $Q$ is isomorphic to $\merge(Q'_1, \dots, Q'_k)$.
\end{proof}

Thanks to \cref{lemma:inductive-construction-rooted-trees}, when it is convenient, we can (and do) assume that rooted tree quivers are constructed inductively using the $\merge$ operation of the result.

Rooted tree quivers over a rooted tree quiver $Q$ can also be defined inductively, as follows.

\begin{definition}
    \label{definition:indutive-rooted-tree-over}
    Let $(Q,\sigma)$ be an inductive rooted tree quiver.
    An \emph{inductive rooted tree quiver over} $Q$ is a pair $(T, f_T : T \to Q)$ such that:
    \begin{itemize}[leftmargin=2cm]
        \item[(Base case)] either $T = \ast$ and $f_T = u_\ast : \ast \to Q$;
        \item[(Ind.~\hspace{0.1cm}case)] or
              $Q = \merge(Q_1, \dots, Q_k)$, $T = \merge\left(T_1^\bullet, \dots, T_k^\bullet\right)$, and $f = \merge(f_1^\bullet, \dots, f_k^\bullet)$, where, for each $1 \leq i \leq k$, we have that $T_i^\bullet = \{T_i^1, \dots, T_i^{\ell_i}\}$ is a list of rooted tree quivers,
              and $f_i^\bullet = \{f_i^1 : T_i^1 \to Q_i\,,\,\, \dots\,\,,\, f_i^{\ell_i} : T_i^{\ell_i} \to Q_i\}$ is a list of rooted tree quiver morphisms.
              Note that the list $T_i^\bullet$ is allowed to be empty, but, in order to keep this case disjoint from the base case, we ask that not all lists be empty.
              The rooted tree quiver~$T$ is constructed using the same operation $\merge$ as in \cref{definition:inductive-rooted-tree-quiver}, that is
              \[
                  \merge\left(T_1^\bullet, \dots T_k^\bullet\right) \coloneqq
                  \merge(T_1^1, \dots, T_1^{\ell_1}, \dots, T_k^1, \dots, T_k^{\ell_k}),
              \]
              resulting in a rooted tree quiver with root $\tau$.
              The morphism $\merge(f_1^\bullet, \dots, f_k^\bullet)$ is uniquely characterized by mapping $\tau$ to $\sigma$, and by coinciding with $f_i^j$ when restricted to $T_i^j$ for every $1 \leq i \leq k$ and $1 \leq j \leq \ell_i$.
    \end{itemize}
\end{definition}

See \cref{figure:inductive-tree-definitions} for an illustration.
In this case, we also refer to the operation $\merge$ as the \emph{gluing} operation.

\begin{lemma}
    \label{lemma:inductive-construction-rooted-trees-over}
    Let $Q$ be an inductive rooted tree quiver.
    Every object of $\rtree_{/Q}$ is isomorphic to an inductive rooted tree quiver over $Q$.
\end{lemma}
\begin{proof}
    The proof proceeds by induction on $Q$ and on the number of vertices of $T$, and is very similar to that of \cref{lemma:inductive-construction-rooted-trees}.
    If $T$ has a single vertex, then $(T, f_T : T \to Q)$ is isomorphic to $(\ast, u_\ast : \ast \to Q)$.
    Otherwise, $Q = \merge(Q_1, \dots, Q_k)$.
    Let $T\setminus \tau$ be obtained by removing the root and all adjacent edges from $T$.
    It is straightforward to see that, for each connected component $C$ of $T\setminus\tau$, the restriction $f_T|_C : C \to Q$ factors as a rooted tree quiver morphism $f_C : C \to Q_{i_C}$, for some $1 \leq i_C \leq k$, followed by the inclusion $Q_i \to Q$, and that the gluing of the rooted tree quivers $(C, f_C : C \to Q_i)$ for all $C$ is isomorphic to $(T, f_T : T \to Q)$.
    By inductive hypothesis, each rooted tree quiver $(C, f_C : C \to Q_i)$ is isomorphic to an inductive rooted tree quiver, so $(T,f_T : T \to Q)$ is isomorphic to an inductive rooted tree quiver over $Q$.
\end{proof}

We can use induction to give a useful characterization of morphisms in $\rtree_{/Q}$.

\begin{lemma}
    \label{remark:inductive-characterization-morphisms}
    Let $Q$ be an inductive rooted tree quiver and let $(S, f_S : S \to Q)$ and $(T, f_T : T \to Q)$ be inductive rooted tree quivers over $Q$.
    Then we are in one of the following three cases:
    \begin{enumerate}[leftmargin=2cm]
        \item If $T = \ast$ but $S \neq \ast$, then $\hom_{\rtree_{/Q}}(S, T) = \emptyset$.
        \item If $S = \ast$, then $\hom_{\rtree_{/Q}}(S, T) = \{u_\ast : \ast \to T\}$.
        \item Otherwise, let $Q = \merge(Q_1, \dots, Q_k)$, $S = \merge(S_1^\bullet, \dots, S_k^\bullet)$, and $T = \merge(T_1^\bullet, \dots, T_k^\bullet)$, with $S_i^\bullet = \{S_i^1, \dots, S_i^{\ell_i}\}$ and $T_i^\bullet = \{T_i^1, \dots, T_i^{m_i}\}$.
              Then $\hom_{\rtree_{/Q}}(S,T)$ is in bijection with the set of pairs $(\nu,g)$, where $\nu$ is a function taking as input $1 \leq i \leq k$ and $1 \leq j \leq \ell_i$, and returning $1 \leq \nu(i,j) \leq m_i$; and $g$ is a function taking as input $1 \leq i \leq k$ and $1 \leq j \leq \ell_i$, and returning a morphism $g_i^j \in \hom_{\rtree_{/Q_i}}(S_i^j, T_i^{\nu(i,j)})$.
    \end{enumerate}
    The bijection in $(3)$ is given as follows.
    Given a morphism $g : S \to T$ in $\rtree_{/Q}$, restrict it to $S$ minus its root to get, for each $1 \leq i \leq k$ and every $1 \leq j \leq \ell_i$, an index $1 \leq \nu(i,j) \leq m_i$ and a morphism $g_i^j : S_i^j \to T_i^{\nu(i,j)}$.

     In particular, in case (3) we get the following recursive formula for counting morphisms in $\rtree_{/Q}$:
    \[
     \left|\hom_{\rtree_{/Q}}(S,T)\right|  =  \prod_{i=1}^k \prod_{j = 1}^{\ell_i} \sum_{n = 1}^{m_i} \left|\hom_{\rtree_{Q_i}}(S_i^j, T_i^n)\right|.
      \]
\end{lemma}
\begin{proof}
    The cases (1) and (2) are straightforward.
    For case (3), note that, for each $1 \leq i \leq k$ and $1 \leq j \leq \ell_i$, the quiver $S_i^j$ is connected, so its image under $g$ must lie entirely in $T_i^n$ for some $1 \leq n \leq m_i$; so that we can only have $\nu(i,j) = n$.
    This shows that the bijection described in the statement is a well-defined function.
    The fact that it is a bijection is a routine check.
    The formula then follows directly from (3).
\end{proof}

\subsection{Reduced rooted tree quivers}
Let $Q$ be a rooted tree quiver.
The recursive definition of reduced rooted trees in Kinser's work is useful for induction, but it is not especially transparent on first reading.
We therefore use the following  formulation as our definition.

\begin{definition}
    \label{definition:reduced}
    A rooted tree quiver $T$ over a rooted tree quiver $Q$ is \emph{reduced} if its only endomorphism over $Q$ is the identity, that is,
    \[
        \hom_{\rtree_{/Q}}(T,T)=\{\id_T\}.
    \]
\end{definition}

The equivalence between this intrinsic condition and Kinser's recursive definition will be established in \cref{proposition:characterization-reduced-trees}, after introducing the preorder $\preceq_Q$ and identifying it with the existence of morphisms over $Q$.

\subsection{The preorder on rooted tree quivers over a rooted tree quiver}
Let $Q$ be a rooted tree quiver.
We now define a preorder on the (isomorphism classes of) objects of $\rtree_{/Q}$.
This is a generalization of an order introduced by Kinser only on a certain subset of the rooted tree quivers over $Q$ (the reduced ones, defined below).
Since the definition is isomorphism invariant, we assume that all rooted tree quivers involved are inductive, thanks to \cref{lemma:inductive-construction-rooted-trees,lemma:inductive-construction-rooted-trees-over}.

\begin{definition}
    \label{definition:order-on-trees}
    Let $S$ and $T$ be rooted tree quivers over $Q$.
    Let $S \preceq_Q T$ if and only if:
    \begin{itemize}[leftmargin=2cm]
        \item[(Base case)] either $S = \ast$;
        \item[(Ind.~\hspace{0.1cm}case)] or the following holds.
              We have $Q = \merge(Q_1, \dots, Q_k)$, $S = \merge(S_1^\bullet, \dots, S_k^\bullet)$, $T = \merge(T_1^\bullet, \dots, T_k^\bullet)$,
              with $S_i^\bullet = \{S_i^1, \dots, S_i^{\ell_i}\}$ and $T_i^\bullet = \{T_i^1, \dots, T_i^{m_i}\}$,
              and for every $1 \leq i \leq k$ and every $1 \leq j \leq \ell_i$, there exists $1 \leq n \leq m_i$ such that $S_i^j \preceq_{Q_i} T_i^n$.
    \end{itemize}
\end{definition}

The preorder just defined is the ingredient entering Kinser's recursive description of reduced rooted trees.
We postpone that description until \cref{proposition:characterization-reduced-trees}, where it follows naturally from the interpretation of $\preceq_Q$ in terms of morphisms.

\subsection{Linearization of rooted tree quivers over a rooted tree quiver}
Any quiver morphism $f: Q' \to Q$ induces a \emph{push-forward} functor $f_* : \rep(Q') \to \rep(Q)$,
where $f_*(V)$ is obtained as follows for each vertex $x \in Q$ and every arrow $\alpha \in Q$:
\[
    f_*(V)_x = \bigoplus_{y\in f_0^{-1}(x)}V_y \qquad \qquad \qquad  f_*(V)_\alpha = \sum_{\beta\in f_1^{-1}(\alpha)}V_\beta
\]

\begin{definition}
    \label{definition:linearization}
    Let $Q$ be a quiver and let $(T, f_T : T \to Q)$ be a quiver over $Q$.
    The \emph{linearization} of $T$, denoted $\kbb_T \in \rep(Q)$, is defined by $\kbb_T \coloneq (f_T)_*(\mathbb{1}_{T})$, the push-forward by $f_T$ of the constant representation $\mathbb{1}_{T} \in \rep(T)$.
\end{definition}

See \cref{figure:inductive-tree-definitions} for an example.

\begin{definition}
    \label{definition:linearized-rooted-tree-quiver}
    Let $Q$ be a rooted tree quiver.
    A \emph{linearized rooted tree quiver} over $Q$ is any representation of $Q$ isomorphic to $\kbb_T$ for $T \to Q$ a rooted tree quiver over $Q$.
\end{definition}

If $Q$ is a rooted tree quiver, linearization of rooted tree quivers over $Q$ can be extended to a functor
$\rtree_{/Q} \to \rep(Q)$.
In order to describe this, we first introduce a gluing construction for representations of rooted tree quivers; this is a particular case of the construction in \cite[Section~2]{ringel-2}.
In this case too, we refer to the operation $\merge$ as the \emph{gluing} operation.

\begin{definition}
    \label{definition:gluing-modules}
    Let $(Q_1,\sigma_1), \dots, (Q_k,\sigma_k)$ be rooted tree quivers.
    \begin{itemize}
    \item Let $M_1^\bullet, \dots, M_k^\bullet$ be (possibly empty) lists such that $M_i^\bullet = \{M_i^1, \dots, M_i^{\ell_i}\}$ is a list of representations of $Q_i$, with the property that they take the value $\kbb$ at the root of $Q_i$.
    Define a representation $M = \Gcal(M_1^\bullet, \dots, M_k^\bullet)$ of $(Q,\sigma) = \merge(Q_1, \dots, Q_k)$ as follows:
    $M(\sigma) = \kbb$, $M|_{Q_i} = \bigoplus_{1 \leq j \leq \ell_i} M_i^j$, and the structure morphism $M_i^j(\sigma_i) \to M(\sigma)$ is the identity $\kbb \to \kbb$, for all $1 \leq i \leq k$.
    \item Moreover, let $N_1^\bullet, \dots, N_n^\bullet$ be (possibly empty) lists such that $N_i^\bullet = \{N_i^1, \dots, N_i^{m_i}\}$ is a list of representations of $Q_i$, with the property that they take the value $\kbb$ at the root of $Q_i$.
    Let us also be given a list of morphisms $\{g_1, \dots, g_k\}$ such that
    \[
        g_i : \bigoplus_{1\leq a \leq m_i} N_i^a \to \bigoplus_{1 \leq b \leq \ell_i} M_i^b.
    \]
    Define $\Gcal(g_1, \dots, g_k) : \Gcal(N_1^\bullet, \dots, N_k^\bullet) \to \Gcal(M_1^\bullet, \dots, M_k^\bullet)$ as the only morphism that restricts to $\sum g_i$ on $Q$ minus the root.
    \end{itemize}
\end{definition}

The following result says that gluing rooted tree quivers and then linearizing is the same as linearizing and then gluing.
The proof is straightforward.

\begin{lemma}
    \label{lemma:merge-linearize}
    Let $Q = \merge(Q_1, \dots, Q_k)$ and $T = \merge(T_1^\bullet, \dots, T_k^\bullet)$, with $T_i^\bullet = \{T_i^1, \dots, T_i^{\ell_i}\}$.
    Define $M_i^\bullet = \{\kbb_{T_i^1}, \dots, \kbb_{T_i^{\ell_i}}\}$.
    Then, $\kbb_{T} \cong \Gcal(M_1^\bullet, \dots, M_k^\bullet)$.
    \qed
\end{lemma}

As is usual, in the next definition we assume that all rooted tree quivers are inductive.

\begin{definition}
    \label{definition:linearization-functor}
    The \emph{linearization functor} $\Lcal_Q : \rtree_{/Q} \to \rep(Q)$ is defined on objects inductively, as follows
    \begin{itemize}[leftmargin=2cm]
        \item Let $\Lcal_Q(\ast)$ be the representation of $Q$ that is $\kbb$ at the root and zero elsewhere.
        \item Let $Q = \merge(Q_1, \dots, Q_k)$, and $T = \merge(T_1^\bullet, \dots, T_k^\bullet)$, with $T_i^\bullet = \{T_i^1, \dots, T_i^{\ell_i}\}$.
        Define
        \[
            \Lcal_Q(T) =
            \Gcal\big(
                \Lcal_{Q_1}(T_1^\bullet)
                , \dots,
                \Lcal_{Q_k}(T_k^\bullet)
                \big),
        \]
        where $\Lcal_{Q_i}(T_i^\bullet) = \{\Lcal_{Q_i}(T_i^1), \dots, \Lcal_{Q_i}(T_i^{\ell_i})\}$, and where we used the inductive hypothesis and the gluing operation for representations.
    \end{itemize}
    The functor $\Lcal$ is defined on a morphism $g : S \to T$ of $\rtree_{/Q}$ also inductively:
    \begin{itemize}[leftmargin=2cm]
        \item If $S = \ast$, the representation $\Lcal_Q(S)$ is $\kbb$ at the root and zero elsewhere, and the morphism $\Lcal_Q(S) \to \Lcal_Q(T)$ is defined to be the identity $\kbb$ at the root and zero elsewhere.
        \item If $S \neq \ast$, then $T \neq \ast$, since there would otherwise not be any morphism $S \to T$.
              Let $Q = \merge(Q_1, \dots, Q_k)$, $S = \merge(S_1^\bullet, \dots, S_k^\bullet)$, and $T = \merge(T_1^\bullet, \dots, T_k^\bullet)$, with $S_i^\bullet = \{S_i^1, \dots, S_i^{\ell_i}\}$ and $T_i^\bullet = \{T_i^1, \dots, T_i^{m_i}\}$.
              Restrict $g : S \to T$ to $T$ minus its root, to get, for each $1 \leq i \leq k$ and $1 \leq j \leq \ell_i$, a $1 \leq n \leq m_i$ and a morphism $g_i^j : S^j_i \to T^n_i$, as in \cref{remark:inductive-characterization-morphisms}(3).
              By inductive hypothesis, we get morphisms $\Lcal_{Q_i}(g_i^j) : \Lcal(S^j_i) \to \Lcal(T^n_i)$ in $\rep(Q_i)$, which we glue using \cref{definition:gluing-modules}(2) to get a morphism $\Lcal_Q(S) \to \Lcal_Q(T)$.
    \end{itemize}
\end{definition}

When there is no ambiguity, we denote the linearization functor $\Lcal_Q$ by $\Lcal$.
The following is proven by a straightforward induction.

\begin{lemma}
    \label{lemma:L-is-push-forward}
    Let $Q$ be a rooted tree quiver, and let $T \to Q$ be a rooted tree quiver over $Q$.
    Then $\kbb_T \cong \Lcal(T) \in \rep(Q)$.
    \qed
\end{lemma}

\subsection{Rooted tree modules}\label{sect::rooted tree modules}

Following \cite{Ri}, we first define the coefficient quiver of a representation with respect to a basis.

\begin{definition}
    \label{definition:coefficient-quiver}
    Let $Q$ be a quiver, and let $V \in \rep(Q)$.
    Given a set of bases $\Bcal = \{\Bcal_x \subseteq V_x\}_{x \in Q_0}$, define the \emph{coefficient quiver} $\Gamma(V;\Bcal)$ of $V$ with respect to $\Bcal$ as follows.
    The set of vertices of $\Gamma(V;\Bcal)$ is given by $\bigsqcup_{x \in Q_0} \Bcal_x$, the disjoint union of all basis elements in $\Bcal$, and we add an arrow from a vertex $b \in \Bcal_x$ to a vertex $b' \in \Bcal_{x'}$ for each arrow $\alpha : x \to x'$ in $Q$ such that there is a non-zero coefficient in column $b$ and row $b'$ of the matrix of $V_\alpha : Q_x \to Q_{x'}$ in the basis $\Bcal$.
\end{definition}

As an example, if $(T, f_T : T \to Q)$ is a quiver over a quiver $Q$, then
the coefficient quiver $\Gamma(\kbb_T, \Bcal)$ is isomorphic to the quiver $T$ itself, where $\kbb_T$ is the linearization of $T$ (\cref{definition:linearization}), and $\Bcal$ is the basis exhibiting each vector space $(\kbb_T)_x$ as freely generated by the vertices $(f_T)^{-1}(x) \subseteq T_0$, for $x \in Q_0$.

\begin{definition}
    \label{definition:rooted-tree-module}
    Let $Q$ be a quiver.
    A representation $V \in \rep(Q)$ is a \emph{rooted tree module} if there exists a basis for $V$ for which the coefficient quiver is a rooted tree quiver.
\end{definition}

Rooted tree modules are a special case of \emph{tree modules} \cite{Ri}, which are representations admitting a basis for which the coefficient quiver is a tree.

Recall that if $Q$ is a rooted tree quiver, and $x \in Q_0$, then $Q_{\leq x}$ is a subquiver of $Q$, which is also a rooted tree quiver.
Let $\iota : Q_{\leq x} \to Q$ denote the inclusion.

\begin{lemma}
    \label{lemma:rooted-tree-module-is-linearized-rooted-tree}
    Let $Q$ be a rooted tree quiver, and let $V \in \rep(Q)$.
    The representation $V$ is a rooted tree module if and only if there exists $x \in Q$ and $(T, f_T : T \to Q_{\leq x})$, a rooted tree quiver over $Q_{\leq x}$, such that $V \cong \iota_*(\kbb_T)$.
\end{lemma}
\begin{proof}
Of course, if $V \cong \iota_*(\kbb_T)$ as in the lemma, then $V$ is a rooted tree module. Conversely, let $x \in Q_0$ be the maximum (in the poset relation $\leq_Q$) of the vertices at which $M$ is non-zero (\cref{lemma:rooted-tree-lattice}).
    Let $\Bcal$ be a basis of $V$ such that $\Gamma(V;\Bcal)$ is a rooted tree quiver.
    As observed in \cite[Proposition~2]{Ri}, there is a base change turning all the non-zero coefficients into 1, thus~$V \cong \iota_*(\kbb_T)$.
\end{proof}

We now introduce, in our language, the reduced representations of Kinser.

\begin{definition}[{cf.~\cite[Definition~16]{K10}}]
    \label{definition:reduced-rooted-tree-module}
    Let $Q$ be a rooted tree quiver, and let $V \in \rep(Q)$.
    The representation $V$ is a \emph{reduced rooted tree module} if there exists $x \in Q$ and $(T, f_T : T \to Q_{\leq x})$, a reduced rooted tree quiver over $Q_{\leq x}$, such that $V \cong \iota_*(\kbb_T)$.
\end{definition}

Equivalently, a reduced rooted tree module over a rooted tree $Q$ is a representation of $Q$ admitting a set of bases with the property that the associated coefficient quiver (\cref{definition:coefficient-quiver}) is a reduced rooted tree quiver over $Q_{\leq x}$, for some vertex $x \in Q$.

\section{Proof of \cref{Theorem A}(\ref{equation:image-h0})}
\label{section:proof-thm-A-1}

This section contains straightforward, yet sometimes tedious, categorical arguments.
We skip some details for conciseness.
We start by recalling the definition of the path category of a quiver.

\begin{definition}
    \label{definition:path-category}
    Let $Q$ be a quiver.
    The \emph{path category} of $Q$ has as objects the set of vertices $Q_0$ of $Q$, and for $x,y \in Q_0$, as set of morphisms $\hom(x,y)$ the set of directed paths (\cref{definition:order-on-quiver}) from $x$ to $y$.
    Composition is given by composition of paths.
\end{definition}

By a standard abuse of notation, we denote the path category of a quiver $Q$ also by $Q$.

\begin{definition}
    Let $Q$ be a rooted tree quiver.
    Let $\rtrees_{/Q}$ denote the full subcategory of $\quiv_{/Q}$ of objects $(T, f_T : T \to Q)$ such that $T$ is a (potentially empty) disjoint union of rooted tree quivers $T = T_1 \sqcup \cdots \sqcup T_k$ ($k \geq 0$), and such that, for all $1 \leq i \leq k$, the restriction of $f_T$ to $T_i$ is a rooted tree quiver morphism (i.e., it is root-preserving).
\end{definition}

Let $\set$ denote the category of finite sets.
We now define a functor $\Sigma : \set^Q \to \rtrees_{/Q}$.

\begin{definition}
    Let $Q$ be a rooted tree quiver, and let $X : Q \to \set$ be a functor from the path category of $Q$.
    The vertex set of ${\Sigma X}$ is the disjoint union $\bigsqcup_{s \in Q_0}X(s)$, and there is an arrow from $a$ to $b$ in ${\Sigma X}$ whenever $X(\alpha)(a) = b$ for some arrow $\alpha$ in $Q$. 
    The function $f_{\Sigma X} : \Sigma X \to Q$ is the quiver morphism determined by $f^{-1}(s) = X(s)$ for all $s \in Q_0$.
    Given a natural transformation $g : X \to Y$, we let $\Sigma g : \Sigma X \to \Sigma Y$ be the quiver morphism determined by mapping a vertex $x$ of $\Sigma X$ corresponding to $x \in X(s)$ to the vertex of $\Sigma Y$ corresponding to $g_s(x) \in Y(s)$.
\end{definition}

We now define a functor $\fib : \rtrees_{/Q} \to \set^Q$.

\begin{definition}
    Let $Q$ be a rooted tree quiver, and let $(T, f_T : T \to Q)$ in $\rtrees_{/Q}$.
    Define $\fib_T : Q \to \set$ by $\fib_T(x) = f_T^{-1}(x)$.
    If $\phi : x \to y$ is a morphism of the path category of $Q$, then $y = \suc^n(x)$ for some $n$.
    Then, if $s \in \fib_T(x)$, define $\fib_T(\phi)(s) = \suc^n(s)$.
\end{definition}

The following is then a straightforward check.

\begin{lemma}
    \label{proposition:set-representation-is-tree-over}
    Let $Q$ be a rooted tree quiver.
    The functors $\Sigma$ and $\fib$ are inverse equivalences of categories.
    \qed
\end{lemma}

The linearization functor $\Lcal : \rtree_{/Q} \to \rep(Q)$ then extends readily to a functor $\Lcal : \rtrees_{/Q} \to \rep(Q)$ by $\Lcal(T) = \Lcal(T_1) \oplus \cdots \oplus \Lcal(T_k)$, where $T =  T_1 \sqcup \cdots \sqcup T_k \in \rtrees_{/Q}$.

If $F : D \to E$ is a functor and $C$ is a category, let $F_* : D^C \to E^C$ denote the functor between functor categories given by post-composition, that is $F_*(G) = F \,\circ\, G$ for $G : C \to D$.

Let $\free : \set \to \vect$ be the free vector space functor.

\begin{lemma}
    \label{lemma:linearization-is-free}
    There is a natural isomorphism of functors
    $\Lcal \cong \free_* \circ \fib : \rtrees_{/Q} \to \rep(Q)$, 
    where $\free_* : \set^Q \to \rep(Q)$.
\end{lemma}
\begin{proof}
    It is sufficient to do this only for rooted tree quivers over $Q$ (as opposed to disjoint unions of these).
    We prove that, for every $(T, f_T : T \to Q)$ rooted tree quiver over $Q$ we have $\Lcal(T) \cong \free_*(\fib_T)$, and omit the naturality proof.
    To prove this, we proceed by induction.
    The case $T = \ast$ is immediate.
    Otherwise, let $T = \merge(T_1^\bullet, \dots, T_k^\bullet)$, with $T_i^\bullet = \{T_i^1, \dots, T_i^{\ell_i}\}$.
    We have 
    \begin{align*}
            \Lcal(T) &=
            \Gcal\left(\bigoplus_{1 \leq j \leq \ell_1} \Lcal(T_1^j), \dots, \bigoplus_{1 \leq j \leq \ell_k} \Lcal(T_k^j) \right)\\
            &\cong \Gcal\left(\bigoplus_{1 \leq j \leq \ell_1} \free_*\left(\fib_{T_1^j}\right), \dots, \bigoplus_{1 \leq j \leq \ell_k} \free_*\left(\fib_{T_k^j}\right) \right)\\
            &\cong \free_*\left(\fib_{\Gcal(T_1^\bullet, \dots, T_k^\bullet)}\right),
    \end{align*}
    where in the equality we used \cref{definition:linearization-functor}, and in the first isomorphism we used the inductive hypothesis.
    The second isomorphism is straightforward to check, using \cref{definition:gluing-modules}.
\end{proof}

As usual, we denote the path-components functor by $\pi_0 : \set \to \top$.
The following is standard.

\begin{lemma}
    \label{lemma:homology-is-pi0}
    There is a natural isomorphism $\free \,\circ\, \pi_0 \cong H_0(-;\kbb) : \top \to \vect$.
    \qed
\end{lemma}

Let $\disc : \set \to \top$ be the functor that endows every finite set with the discrete topology.
The proof of the following result is straightforward.

\begin{lemma}
    \label{lemma:disc-pi0-is-identity}
    There is a natural isomorphism $\pi_0 \circ \disc \cong \id_{\set} : \set \to \set$.
    \qed
\end{lemma}

\begin{proof}[Proof of \cref{Theorem A}(\ref{equation:image-h0})]
    Consider the following diagram of categories and functors:
    \[
        \begin{tikzpicture}
            \matrix (m) [matrix of math nodes,row sep=5em,column sep=4em,minimum width=2em,nodes={text height=1.75ex,text depth=0.25ex}]
            {
                \top^Q & \set^Q  & \rtrees_{/Q}             \\
                       & \rep(Q) & \\};
            \path[line width=0.75pt, -{>[width=8pt]}]
            (m-1-1) edge [bend right] [left] node {$H_0\,\,$} (m-2-2)
            (m-1-2) edge [left] node {$\free_*$} (m-2-2)
            (m-1-3) edge [bend left] [left,above] node {$\Lcal\,\,\,\,$} (m-2-2)
            (m-1-1) edge [above] node {$(\pi_0)_*$} (m-1-2)
            (m-1-3) edge [above] node {$\fib$} (m-1-2)
            ;
        \end{tikzpicture}
    \]
    \cref{lemma:homology-is-pi0,lemma:disc-pi0-is-identity} imply that the essential image of $H_0 : \top^Q \to \rep(Q)$ is equal to the essential image of $\free_* : \set^Q \to \rep(Q)$, and \cref{proposition:set-representation-is-tree-over,lemma:linearization-is-free} imply that the essential image of $\free_*$ is equal to the essential image of $\Lcal$, which, by \cref{lemma:L-is-push-forward}, consists of all direct sums of linearized rooted tree quivers over $Q$.
\end{proof}

\section{Proofs of \cref{Theorem B} and \cref{Theorem A}(\ref{equation:add-image-h0})}

Let $Q$ be a rooted tree quiver.
In this section we partly redevelop and generalize Kinser's theory \cite{K10} using our induction methods, which simplify exposition.

We first relate the preorder $\preceq_Q$ to the existence of morphisms of rooted trees and then deduce Kinser's recursive characterization of the reduced trees defined intrinsically in \cref{definition:reduced}.

We also prove the elder rule (\cref{proposition:split-off-summand}), which is the key ingredient in the proof of \cref{Theorem B} and in the algorithms of \cref{thm::algorithm}.

We give the proofs of \cref{Theorem B}, \cref{corollary:finite-representation-type}, and \cref{Theorem A}(2) in that order, since the proof of \cref{Theorem A}(2) relies on \cref{corollary:finite-representation-type}, which in turn relies on \cref{Theorem A}(1).

\begin{proposition}[{cf.~\cite[Proposition~9~and~Theorem~18]{K10}}]
    \label{proposition:order-graphmap-stablemap}
    Let $(Q,\sigma)$ be a rooted tree quiver, and let $S$ and $T$ be rooted tree quivers over $Q$.
    The following are equivalent:
    \begin{enumerate}
        \item We have $S \preceq_Q T$.
        \item There exists a morphism from $S$ to $T$ in $\quiv_{/Q}$.
        \item There exists a morphism $\phi : \kbb_S \to \kbb_T$ which is non-zero at the root, that is, such that $\phi_\sigma : (\kbb_S)_\sigma \to (\kbb_T)_\sigma$ is non-zero.
    \end{enumerate}
\end{proposition}
\begin{proof}
    Since all of the conditions are isomorphism invariant, we can assume that all rooted tree quivers involved are inductive, thanks to \cref{lemma:inductive-construction-rooted-trees,lemma:inductive-construction-rooted-trees-over}.
    The equivalence $(1) \Leftrightarrow (2) \Leftrightarrow (3)$ is then proven by induction.

    If $S = \ast$ or $T = \ast$, this is clear.
    Otherwise, let $Q = \merge(Q_1, \dots, Q_k)$, $S = \merge(S_1^\bullet, \dots, S_k^\bullet)$, and $T = \merge(T_1^\bullet, \dots, T_k^\bullet)$, with $S_i^\bullet = \{S_i^1, \dots, S_i^{\ell_i}\}$ and $T_i^\bullet = \{T_i^1, \dots, T_i^{m_i}\}$.

    \smallskip

    \noindent $(1) \Rightarrow (2)$.
    Assume that $S \preceq_Q T$.
    Then, for every $1 \leq i \leq k$ and every $1 \leq j \leq \ell_i$, there exists $1 \leq n \leq m_i$ such that $S_i^j \preceq_{Q_i} T_i^n$.
    By inductive hypothesis, there exists a morphism $S_i^j \to T_i^n$ in $\quiv_{/Q_i}$.
    A morphism $S \to T$ in $\quiv_{/Q}$ is then constructed by simply combining all these morphisms, as in \cref{remark:inductive-characterization-morphisms}(3).

    \smallskip

    \noindent $(2) \Rightarrow (3)$.
    This implication does not require the inductive hypothesis, and just follows by applying the linearization functor $\Lcal : \quiv_{/Q} \to \rep(Q)$ to the quiver morphism $S \to T$ over $Q$, and using \cref{lemma:L-is-push-forward}.

    \smallskip

    \noindent $(3) \Rightarrow (1)$.
    Assume that there exists a morphism $\phi : \kbb_S \to \kbb_T$ that is non-zero at the root.
    The morphism $\phi$ induces, by restriction to $Q \setminus \sigma$, a morphism $\phi_{i,j,n} : \kbb_{S_i^j} \to \kbb_{T_i^n}$, for each $1 \leq i \leq k$, $1 \leq j \leq \ell_i$, and $1 \leq n \leq m_i$.
    Since $\phi$ is non-zero at the root, for every $1 \leq i \leq k$ and every $1 \leq j \leq \ell_i$, there exists $1 \leq n \leq m_i$ such that $\phi_{i,j,n}$ is non-zero at the root of $Q_i$.
    By inductive hypothesis, this implies that $S_i^j \preceq_{Q_i} T_i^n$, and thus $S \preceq_Q T$, by definition.
\end{proof}

The preceding characterization of the preorder makes the intrinsic notion of reducedness in \cref{definition:reduced} equivalent to Kinser's recursive condition.

\begin{proposition}[{Kinser's recursive characterization, cf.~\cite[Definition~7]{K10}}]
\label{proposition:characterization-reduced-trees}
    Let $Q$ be a rooted tree quiver and let $(T,f_T:T\to Q)$ be a rooted tree quiver over $Q$.
    After replacing $Q$ and $T$ by isomorphic inductive rooted tree quivers, exactly one of the following occurs:
    \begin{itemize}[leftmargin=2cm]
        \item[(Base case)] $T=\ast$, in which case $T$ is reduced;
        \item[(Ind.~\hspace{0.1cm}case)] we have
        \[
            Q=\merge(Q_1,\dots,Q_k),
            \qquad
            T=\merge(T_1^\bullet,\dots,T_k^\bullet),
        \]
        with $T_i^\bullet=\{T_i^1,\dots,T_i^{\ell_i}\}$, and $T$ is reduced if and only if
        \begin{enumerate}
            \item every $T_i^j$ is reduced as a rooted tree quiver over $Q_i$; and
            \item for every fixed $i$, the rooted tree quivers $T_i^j$ and $T_i^{j'}$ are incomparable with respect to $\preceq_{Q_i}$ whenever $j\neq j'$.
        \end{enumerate}
    \end{itemize}
\end{proposition}
\begin{proof}
    Since both the intrinsic condition of \cref{definition:reduced} and the recursive conditions in the statement are invariant under isomorphism, we may assume that $Q$ and $T$ are inductive, by \cref{lemma:inductive-construction-rooted-trees,lemma:inductive-construction-rooted-trees-over}.
    The case $T=\ast$ is immediate, so suppose that
    \[
        Q=\merge(Q_1,\dots,Q_k),
        \qquad
        T=\merge(T_1^\bullet,\dots,T_k^\bullet).
    \]

    Assume first that $T$ is reduced, so
    \[
        \hom_{\rtree_{/Q}}(T,T)=\{\id_T\}.
    \]
    If some $T_i^j$ admitted a non-identity endomorphism over $Q_i$, then the recursive description of morphisms in \cref{remark:inductive-characterization-morphisms} would extend it, together with the identity maps on all other branches, to a non-identity endomorphism of $T$.
    Hence every $T_i^j$ is reduced.
    Likewise, if $T_i^j\preceq_{Q_i}T_i^{j'}$ for some $j\neq j'$, then \cref{proposition:order-graphmap-stablemap} gives a morphism $T_i^j\to T_i^{j'}$ over $Q_i$, and \cref{remark:inductive-characterization-morphisms} again produces a non-identity endomorphism of $T$.
    Thus distinct branches over the same $Q_i$ are pairwise incomparable.

    Conversely, assume that every $T_i^j$ is reduced and that distinct branches over a fixed $Q_i$ are pairwise incomparable.
    By \cref{proposition:order-graphmap-stablemap}, there is no morphism $T_i^j\to T_i^{n}$ when $n\neq j$.
    Since $T_i^j$ is reduced,
    \[
        \hom_{\rtree_{/Q_i}}(T_i^j,T_i^j)=\{\id_{T_i^j}\}.
    \]
    Therefore the counting formula in \cref{remark:inductive-characterization-morphisms} gives
    \[
        \left|\hom_{\rtree_{/Q}}(T,T)\right|
        =
        \prod_{i=1}^k\prod_{j=1}^{\ell_i}
        \sum_{n=1}^{\ell_i}
        \left|\hom_{\rtree_{/Q_i}}(T_i^j,T_i^n)\right|
        =1.
    \]
    Hence the identity is the only endomorphism of $T$, so $T$ is reduced by \cref{definition:reduced}.
\end{proof}

The next result gives a sufficient condition to split off a summand of the linearization of a gluing of reduced trees: this can be done as soon as one of the trees being glued is smaller, in Kinser's preorder, to another of the trees being glued.
For a related result (not stronger or weaker), see~\cite[Lemma~1]{KM}.

\begin{proposition}[Elder rule]
    \label{proposition:split-off-summand}
    Let $(Q, \sigma) = \merge(Q_1, \dots, Q_k)$ be an inductive rooted tree quiver, and let $T$ be an inductive rooted tree quiver over $Q$ such that $T = \merge(T_1^\bullet, \dots, T_k^\bullet)$, with $T_i^\bullet = \{T_i^1, \dots, T_i^{\ell_i}\}$.
    Suppose that there exists $1 \leq i \leq k$ and $1 \leq j,j' \leq \ell_i$, with $j \neq j'$ and $T_i^j \preceq_{Q_i} T_i^{j'}$.
    Then
    \[
        \kbb_T \cong \kbb_{S} \oplus \iota_*(\kbb_{T_i^j}),
    \]
    where $\iota : Q_i \to Q$ is the inclusion, and where $S = \merge(S_1^\bullet, \dots, S_k^\bullet)$, $S_n^\bullet = T_n^\bullet$ if $n \neq i$, and $S_i^\bullet = T_i^\bullet \setminus \{T_i^j\}$ otherwise.
\end{proposition}
\begin{proof}
    Without loss of generality, we may assume that $i = 1$, $j = 1$, and $j' = 2$.
    By \cref{proposition:order-graphmap-stablemap}, there exists a morphism $T_1^1 \to T_1^2$ in $\quiv_{/Q_1}$, which, after linearization, results in a morphism $\phi : \kbb_{T_1}^1 \to \kbb_{T_1}^2$ that is the identity $\kbb \to \kbb$ at the root of $Q_i$.
    Using this morphism, we could now complete the proof by using \cite[Theorem~3.8]{alonso-kerber}; instead, to be self-contained, we give an explicit isomorphism from $\iota_*(\kbb_{T_1^1}) \oplus \kbb_{S}$ to $\kbb_T$.

    Note that such a morphism $\iota_*(\kbb_{T_1^1}) \oplus \kbb_{S} \to \kbb_T$ is completely determined by its restriction to the root $\sigma$, as well as by its restriction to $Q \setminus \sigma = \bigsqcup_{1 \leq n \leq k} Q_n$, where both modules decompose as $\bigoplus_{1 \leq m \leq k} \bigoplus_{1 \leq n \leq \ell_m} \kbb_{T_m^n}$.
    Thus, we can use block matrix notation to define the morphism as follows:
    \[
        \bordermatrix{
        & \kbb_{T_1^1}  & \kbb_{T_1^2} & \kbb_{T_1^3} & \cdots & \kbb_{T_1^{\ell_1}} & \cdots & \kbb_{T_k^1} & \cdots & \kbb_{T_k^{\ell_k}} & \kbb_{\sigma} \cr
        \kbb_{T_1^1}          & \id           & 0     &  0    & & 0 & & 0 & & 0 & 0 \cr
        \kbb_{T_1^2}          & -\phi         & \id   &  0    & & 0 & & 0 & & 0 & 0 \cr
        \kbb_{T_1^3}          & 0             & 0     &   \id & & 0 & & 0 & & 0 & 0 \cr
        \;\;\vdots & & & & \ddots & & & & & & \cr
        \kbb_{T_1^{\ell_1}}   & 0             & 0     &  0    & & \id & & 0 & & 0 & 0 \cr
        \;\;\vdots & & & & & & \ddots & & & & \cr
        \kbb_{T_k^1}          & 0             & 0     &  0    & & 0 & & \id & & 0 & 0 \cr
        \;\;\vdots & & & & & & & &  \ddots& & \cr
        \kbb_{T_k^{\ell_k}}   & 0             & 0     &   0   & & 0 & & 0 &  & \id  & 0 \cr
        \kbb_{\sigma}         & 0             & 0     &  0    &  & 0 & & 0 &  & 0 & \id
        }
    \]
    This morphism is well-defined since this only needs to be checked at the root, where it is well-defined thanks to the fact that as structure morphisms out of $\kbb_{T_1^1}$ we are using $\id - \phi$, which can be extended to the root as $0$.
    The morphism is an isomorphism since the inverse can be defined using a matrix with the same form, but with $\phi$ instead of $-\phi$.
\end{proof}

The following technical result is useful when applying the elder rule (\cref{proposition:split-off-summand}) inductively.

\begin{lemma}
    \label{lemma:merging-decomposition}
    Let $Q = \merge(Q_1, \dots, Q_k)$ be a rooted tree quiver, and let $T = \merge(T_1^\bullet, \dots, T_k^\bullet)$, with $T_i^\bullet = \{T_i^1, \dots, T_i^{\ell_i}\}$.
    Assume that there exists $1 \leq i \leq k$ and $1 \leq j \leq \ell_i$ such that $\kbb_{T_i^j} \cong \kbb_{T''} \oplus N \in \rep(Q_i)$, for $T'' \to Q_i$ a rooted tree quiver over $Q_i$.
    Then $\kbb_T = \kbb_{T'} \oplus N$, where $T' = \merge(S_1^\bullet, \dots, S_k^\bullet)$ with $S_m^\bullet = \{T_m^1, \dots, T_m^{\ell_m}\}$, except for $m = i$, where we define $S_i^\bullet = (T_i^\bullet \setminus T_i^j) \cup \{T''\}$.
\end{lemma}
    In the statement and proof, we implicitly use the fully faithful functor $\rep(Q_i) \to \rep(Q)$ given by extending representations by zeros.
    This allows us to interpret $N$ as a representation of $Q$.
\begin{proof}
    Note that $N$ in the statement must have the property that $N(\tau_i) = 0$, where $\tau_i$ is the root of $T_i^j$.
    Let $M_i^\bullet = \{\kbb_{T_i^1}, \dots, \kbb_{T_i^{\ell_i}}\}$ (as in \cref{lemma:merge-linearize}), and let
    $L_i^\bullet = \{\kbb_{S_i^1}, \dots, \kbb_{S_i^{\ell_i}}\}$.
    Then, we have
    \[
        \kbb_T \cong \Gcal(M_1^\bullet, \dots, M_k^\bullet) \cong \Gcal(L_1^\bullet, \dots, L_k^\bullet) \oplus N \cong \kbb_{T'} \oplus N,
    \]
    where in the first and third isomorphism we used \cref{lemma:merge-linearize}, and in the second isomorphism we used the fact that $N(\tau_i) = 0$, so it does not interact in the gluing operation of modules (\cref{definition:gluing-modules}).
\end{proof}

\begin{lemma}
    \label{lemma:inductive-step-in-decomposition}
    Let $Q$ be a rooted tree quiver, and let $(T, f_T : T \to Q)$ be a rooted tree quiver over $Q$.
    If $T$ is not reduced, then there exists a reduced rooted tree module $N$ over $Q$ and a rooted tree quiver $T' \to Q$ over $Q$ such that $\kbb_T \cong \kbb_{T'} \oplus N$.
\end{lemma}
\begin{proof}
    Without loss of generality, we may assume that $Q$ is an inductive rooted tree quiver and that $T$ is an inductive rooted tree quiver over $Q$, and proceed by induction.

    Since $T$ is not reduced, \cref{proposition:characterization-reduced-trees} gives
    $Q = \merge(Q_1, \dots, Q_k)$ and
    $T = \merge(T_1^\bullet, \dots, T_k^\bullet)$, with
    $T_i^\bullet = \{T_i^1, \dots, T_i^{\ell_i}\}$, and shows that there are two cases to consider.

    If there exist $1 \leq i \leq k$ and $1 \leq j,j' \leq \ell_i$ with $j \neq j'$, such that $T_i^j \preceq_{Q_i} T_i^{j'}$, then $\kbb_T$ decomposes as required, by \cref{proposition:split-off-summand}.

    If there exists $1 \leq i \leq k$ and $1 \leq j \leq \ell_i$ such that $T_i^j$ is not reduced, then, by inductive hypothesis, $\kbb_{T_i^j} \cong \kbb_{T''} \oplus N$, with $N$ a reduced rooted tree module and $T'' \to Q_i$ a rooted tree quiver over $Q_i$.
    \cref{lemma:merging-decomposition} then finishes the proof.
\end{proof}

We now reprove a theorem of Kinser's, which, in our language, gives a combinatorial characterization of the indecomposable linearized rooted tree quivers over a rooted tree quiver $Q$.
Thanks to \cref{lemma:rooted-tree-module-is-linearized-rooted-tree}, it is enough to give this characterization for linearizations of rooted tree quivers over $Q$.

\begin{theorem}[{cf.~\cite[Corollary~19]{K10}}]
    \label{theorem:characterization-indecomposable}
    Let $Q$ be a rooted tree quiver, and let $(T, f_T : T \to Q)$ be a rooted tree quiver over $Q$.
    The representation $\kbb_T \in \rep(Q)$ is indecomposable if and only if $T$ is reduced.
\end{theorem}
\begin{proof}
    If $T$ is not reduced, then it is not indecomposable, by \cref{lemma:inductive-step-in-decomposition}.
    Assume now that $T$ is reduced.
    We will prove that $\End(\kbb_T)$ is a local ring, which implies that $\kbb_T$ is indecomposable
    \cite[Corollary I.4.8 (a)]{ASS}.
    First, let $I \subseteq \End(\kbb_T)$ be the ideal of morphism $\kbb_T \to \kbb_T$ which are zero at the root.
    To prove that this ideal is the only maximal ideal of $\End(\kbb_T)$ (and thus that $\End(\kbb_T)$ is local), we prove that every element $\phi \in \End(\kbb_T) \setminus I$ is invertible.
    We proceed by induction.
    If $T$ has a single vertex, then $\kbb_T$ is clearly indecomposable.
    Otherwise, $Q = \merge(Q_1, \dots, Q_k)$, and $T = \merge(T_1^\bullet, \dots, T_k^\bullet)$, with $T_i^\bullet = \{T_i^1, \dots, T_i^{\ell_i}\}$.

    By \cref{proposition:characterization-reduced-trees}, every $T_i^j$ is reduced and distinct branches over a fixed $Q_i$ are pairwise incomparable.
    The morphism $\phi$ induces, by restriction to $Q$ minus its root $\sigma$, a morphism $\phi_{i,j,n} : \kbb_{T_i^j} \to \kbb_{T_i^n}$, for each $1 \leq i \leq k$ and $1 \leq j,n \leq \ell_i$.
    By \cref{proposition:order-graphmap-stablemap}($3 \Rightarrow 1$) and the pairwise incomparability, the morphism $\phi_{i,j,n}$ must be zero unless $j = n$, so it suffices to prove that $\phi_{i,j,j} : \kbb_{T_i^j} \to \kbb_{T_i^j}$ is invertible for all $1 \leq j \leq \ell_i$.

    The morphism $\phi_{i,j,j}$ is non-zero at the root of $Q_i$, since $\phi$ is non-zero at the root of $Q$, so, by inductive hypothesis, it is invertible, concluding the proof.
\end{proof}

\begin{proof}[Proof of \cref{Theorem B}]
    This follows directly from \cref{lemma:inductive-step-in-decomposition} and induction.
\end{proof}

\begin{proof}[Proof of \cref{corollary:finite-representation-type}]
    We first show that there are finitely many isomorphism classes of reduced rooted tree modules over $Q$.
    By \cref{definition:reduced-rooted-tree-module}, this is equivalent to showing that there are finitely many isomorphism classes of reduced rooted tree quivers over any given rooted tree quiver.

    This follows from Kinser's recursive characterization in \cref{proposition:characterization-reduced-trees} and induction over rooted tree quivers over rooted tree quivers (\cref{section:inductive-rooted-tree-quivers}).

    To conclude, we must show that every representation in $\rep_{H_0}(Q)$ decomposes as a direct sum of reduced rooted tree modules over $Q$, since the reduced rooted tree modules are indecomposable (\cref{theorem:characterization-indecomposable}), and every representation in $\rep(Q)$ decomposes uniquely as a direct sum of indecomposables.
    By \cref{Theorem A}(1) and \cref{lemma:rooted-tree-module-is-linearized-rooted-tree}, every representation in $\rep_{H_0}(Q)$ is a direct sum of rooted tree modules over $Q$, so the result follows from \cref{Theorem B}.
\end{proof}

\begin{proof}[Proof of \cref{Theorem A}(\ref{equation:add-image-h0})]
    The inclusion $(\subseteq)$ follows directly from \cref{corollary:finite-representation-type}.
    To prove the inclusion $(\supseteq)$ it is sufficient to show that, for every rooted tree module $M \in \rep(Q)$, there exists representations $V \in \rep_{H_0}(Q)$ and $A \in \rep(Q)$ such that $V \cong M \oplus A$.

    By \cref{lemma:rooted-tree-module-is-linearized-rooted-tree}, there exists $x \in Q_0$ and a rooted tree quiver $T \to Q_{\leq x}$ over $Q_{\leq x}$ such that $\iota_*(\kbb_T) \cong M$, where $\iota : Q_{\leq x} \to Q$ is the inclusion.
    If $x$ is the root of $Q$, then $\kbb_T \in \rep_{H_0}(Q)$ by \cref{Theorem A}(\ref{equation:image-h0}), and $\kbb_T \cong M$, which proves the claim in this case.

    Otherwise, there exist elements $y_1, \dots, y_k \in Q$ such that
    \[
        Q_{\leq \suc(x)} = \merge(Q_{\leq x}, Q_{\leq y_1}, \dots, Q_{\leq y_k}).
    \]
    Define $S_1 = \merge(\{T\}, \emptyset, \dots, \emptyset)$ and $T_1 = \merge(\{T, T\}, \emptyset, \dots, \emptyset)$, which are rooted tree quivers over $Q_{\leq \suc(x)}$.
    By \cref{proposition:split-off-summand}, we have $\kbb_{T_1} \cong \kbb_T \oplus \kbb_{S_1} \in \rep(Q_{\leq \suc(x)})$.
    If $\suc(x)$ is the root of $Q$, then we are done.
    Otherwise, define $S_{n+1} = \merge(\{S_n\}, \emptyset, \dots, \emptyset)$ and $T_{n+1} = \merge(\{T_n\}, \emptyset, \dots, \emptyset)$ for all $1 \leq n \leq d-1$, where $d \in \Nbb$ is such that, and note that, if $\suc^d(x)$ is the root of $Q$.
    Then $\kbb_{T_d} \cong \iota_*(\kbb_T) \oplus \kbb_{S_d}$ by \cref{lemma:merging-decomposition} and induction, which finishes the proof.
\end{proof}

\section{Algorithms} \label{sect: Algorithms}

\begin{figure}
\begin{algorithm}[H]
    \caption{Decompose linearized rooted tree quiver over $Q$}
    \label{Algorithm 1}
    \SetAlgoLined
    \DontPrintSemicolon
    \SetKwInput{Input}{Input}
    \SetKwInput{Output}{Output}
    \Input{$(T,f_T : T \to Q)$ rooted tree quiver over a rooted tree quiver $Q$.}
    \Output{$S \subseteq T$, such that $\kbb_S \cong \kbb_T \in \rep(Q)$ and such that $\kbb_{C}$ is indecomposable for every connected component $C \subseteq S$.}
    $S$ $\gets$ $T$\;
    $R$ $\gets$ empty relation on vertices of $S$\;
    \For{$\ell \in \{\height(T), \height(T)-1, \dots, 1, 0\}$}{
        \For{$x$ vertex of $S$ of level $\ell$}{
            $\Acal$ $\gets$ maximal set of pairwise $R$-incomparable predecessors of $x$ in $T$, such that every predecessor $x'$ of $x$ in $T$ satisfies $x' R a$ for some $a \in \Acal$\;
            \For{$p \in \pred(x) \setminus \Acal$}{
                delete edge $p \to x$ from $S$\;
            }
        }
        \For{$x,y$ vertices of $S$ of level $\ell$}{
            \If{for every $p \in \pred(x)$, there exists $q \in \pred(y)$, such that $p R q$}{$R$ $\gets$ $R \cup (x,y)$ \Comment{Add $x R y$ to the relation} \;
            }
        }
    }
    \Return $S$
\end{algorithm}

\vspace{0.5cm}

\begin{algorithm}[H]
    \caption{Decompose zero-dimensional homology over $Q$}
    \label{Algorithm 2}
    \SetAlgoLined
    \DontPrintSemicolon
    \SetKwInput{Input}{Input}
    \SetKwInput{Output}{Output}
    \Input{A graph $G = (V,E)$ and a $Q$-filtration $h = (h_V : V \to Q_0, h_E : E \to Q_0)$ with $Q$ a rooted tree quiver.}
    \Output{$(S, f_S : S \to Q)$ quiver over $Q$, such that $\kbb_S \cong H_0(h) \in \rep(Q)$, and such that $\kbb_{C}$ is indecomposable for every connected component $C \subseteq S$.}
    $S$ $\gets$ empty quiver\;
    $(T, f_T : T \to Q)$ $\gets$ \cref{Algorithm 2-aux} on $(G, h)$\;
    \For{$T'$ connected component of $T$}{
        $(S', f_{S'} : S' \to Q)$ $\gets$ \cref{Algorithm 1} on $(T', f_T|_{T'})$\;
        $(S, f_S : S \to Q) \gets (S \sqcup S', f_S \sqcup f_{S'} : S \sqcup S' \to Q)$\;
    }
    \Return $S$
\end{algorithm}

\vspace{0.5cm}

\begin{algorithm}[H]
    \caption{Filtration to disjoint union of rooted tree quivers over $Q$}
    \label{Algorithm 2-aux}
    \SetAlgoLined
    \DontPrintSemicolon
    \SetKwInput{Input}{Input}
    \SetKwInput{Output}{Output}
    \Input{A graph $G = (V,E)$ and a $Q$-filtration $h = (h_V : V \to Q_0, h_E : E \to Q_0)$ with $Q$ a rooted tree quiver.}
    \Output{An object $T \to Q$ of $\rtrees_{/Q}$ such that $\kbb_T \cong H_0(h) \in \rep(Q)$.}
    $T$ $\gets$ empty quiver\;
    $f_T$ $\gets$ empty function to $Q_0$\;
    $\uf$ $\gets$ empty union find\;
    $\mathsf{max}_\ell$ $\gets$ largest level $\ell$ of $Q$ such that $G$ has vertices of level $\ell$\;
    \For{$\ell \in \{\mathsf{max}_\ell, \mathsf{max}_\ell-1, \dots, 1, 0\}$}{
        \For{$v$ vertex of $G$ of level $\ell$}{
            $\uf.\adduf(v)$\;
        }
        \For{$\{v,w\}$ edge of $G$ of level $\ell$}{
            $\uf.\union(v,w)$\;
        }
        $T_0^\ell$ $\gets$ $\uf.\representatives()$\;
        add $T_0^\ell \times \{\ell\}$ to the vertices of $T$
        \Comment{Product with $\{\ell\}$ to make it disjoint}
        \;
        \For{$v \in T_0^\ell$}{
            $f_T(v) \coloneqq f_V(v)$\;
        }
        \For{$v \in T_0^{\ell+1}$}{
            $w$ $\gets$ $\uf.\find(v)$\;
            add $\big((v,\ell+1), (w,\ell)\big)$ to the edges of $T$\;
        }
    }
    \Return $(T, f_T : T \to Q)$
\end{algorithm}
\end{figure}

In this section we describe an algorithm to decompose the linearization of a rooted tree quiver over a rooted tree quiver (\cref{Algorithm 1}), and an algorithm to decompose the zero-dimensional persistent homology of a graph filtered by a rooted tree poset (\cref{Algorithm 2}).
We prove in \cref{proposition:algo-1-proof,proposition:algo-2-proof} that these algorithms are correct, and run in quadratic time.

\subsection{Decomposing linearized rooted tree quiver over a rooted tree quiver}

The following basic definitions are required for \cref{Algorithm 1}.

Let $(T,\tau)$ be a rooted tree quiver.
The \emph{level} $\ell(x)$ of a vertex $x \in T_0$ is the length of the (unique) directed path from $x$ to the root $\tau$.
The \emph{levels} of $T$ are then $\ell(T_0) \subseteq \Nbb$, and the \emph{height} $\height(T)$ of $T$ is the maximum among its levels.

\begin{proposition}
    \label{proposition:algo-1-proof}
    Let $Q$ be a rooted tree quiver.
    Given $T$, a rooted tree quiver over $Q$, \cref{Algorithm 1} runs in $O(|T|^2)$ time, and returns a subquiver $S \subseteq T$ such that $\kbb_S \cong \kbb_T \in \rep(Q)$, and such that $\kbb_{S'}$ is indecomposable for each connected component $S' \subseteq S$.
\end{proposition}
\begin{proof}
    \noindent \emph{Correctness.}
    We consider the following invariant.

    \smallskip

    \noindent\textit{Invariant.} At the end of the $\ell$th iteration of the for-loop of line 3 the following is satisfied.
    \begin{enumerate}
        \item The subquiver $S \subseteq T$ is such that $\kbb_S \cong \kbb_T \in \rep(Q)$.
        \item The restriction $S^\ell \to Q^\ell$ of $f_T|_{S}$ to levels $\ell$ and higher has the property that every connected component of $S^\ell$ is a rooted tree quiver $(C,c)$, and $C \to Q_{\leq f_T(c)}$ is a reduced rooted tree quiver over $Q_{\leq f_T(c)}$.
        \item For $x,y$ vertices of $S$ of level $\ell$ we have that $x R y$ if and only if $f_T(x) = f_T(y) \eqqcolon z$ and $(S_{\leq x}) \preceq_{(Q_{\leq z})} (S_{\leq y})$.
    \end{enumerate}

    \smallskip

    If the invariant is preserved, then the algorithm is correct by conditions (1, 2) of the invariant and the fact that linearized reduced rooted tree quivers are indecomposable (\cref{theorem:characterization-indecomposable}).

    Let us prove that the invariant is preserved.
    For this, we can actually take as base case $\ell = \height(T)+1$, where $S = T$, $S^{\ell} = \emptyset$, and all conditions are immediate to check.
    So we need to prove that, if conditions (1, 2, 3) are satisfied at the end of the $\ell+1$st iteration, then they are satisfied at the end of the $\ell$th iteration.

    We start by checking (2).
    Let $S$ be the value of the variable $S$ at the end of the $\ell+1$st iteration, and let $S'$ be the value of the variable $S$ at the end of the $\ell$th iteration.
    The vertex $x$ of $S$ at level $\ell$ (line 4) induces a subtree quiver $S_{\leq x}$.
    If $x$ has no predecessors, then this tree quiver is the trivial rooted tree quiver, and condition (1) is met at the end of the iteration.
    Otherwise, $S_{\leq x} \cong \Gcal(S_1^\bullet, \dots, S_k^\bullet)$, with $S_i^\bullet = \{S_i^1, \dots, S_i^{\ell_i}\}$ rooted tree quivers over $Q_i$, where $Q_{\leq f(x)} \cong \Gcal(Q_1, \dots, Q_k)$.
    What line 5 does (by condition (3) in the inductive hypothesis) is, for each $1 \leq i \leq k$, choose a maximal set $A_i^\bullet$ of pairwise $\preceq_{Q_i}$-incomparable elements of $S_i^\bullet$ such that every element of $S_i^\bullet$ is smaller (in the order $\preceq_{Q_i}$) than an element of $A_i^\bullet$.
    This guarantees that $A = \Gcal(A_1^\bullet, \dots, A_k^\bullet)$ is a reduced rooted tree quiver over $Q_{\leq f(x)}$ (by condition (2) in the inductive hypothesis), which implies that condition (2) is met at the end of the $\ell$th iteration.

    We now check condition (1).
    From the elder rule (\cref{proposition:split-off-summand}), it follows that $\kbb_S$ is isomorphic to $\kbb_A$ direct sum the linearization of all trees that did not make it into $A$;
    this is because every tree that did not make it into $A$ is smaller (in one of the orders $\preceq_{Q_i}$) to a tree that is in $A$.
    Note that, by construction, the representation $\kbb_A$ is isomorphic to $\kbb_{S'}$:
    \[
        \kbb_S \cong \kbb_A \oplus \left(\bigoplus_{1 \leq i \leq k} \; \bigoplus_{V \in S_i^\bullet \setminus A_i^\bullet} \; \kbb_{V}\right) \cong \kbb_{S'}.
    \]
    This implies that condition (1) is met at the end of the $\ell$th iteration.

    To conclude this part of the proof, we check condition (3).
    Line 9 is checking precisely the condition defining the preorder $\preceq$ (\cref{definition:order-on-quiver}), so condition (3) is met at the end of the $\ell$th iteration thanks to condition (3) in the inductive hypothesis.

    \smallskip

    \noindent \emph{Complexity.}
    Here, whenever $x \in T_0 = S_0$ and we write $\pred(x)$, we mean predecessors in $T$.
    Since $S \subseteq T$, the predecessors in $T$ contain the predecessors in $S$.

    Line 5 can be implemented by simply performing all pairwise comparisons, which takes $O\left(|\pred(x)|^2\right)$ time.
    Lines 6 and 7 take $O\left(|\pred(x)|\right)$ time.
    Thus, the contribution of lines 5, 6, and 7 across all iterations is $O\left(\sum_{x \in T_0} |\pred(x)|^2\right)$ time.
    The contribution of lines 8, 9, and 10 across all iterations is
    \[
        O\left(
            \sum_{\ell = 0}^{\height(T)} \sum_{x,y \text{ of level $\ell$}} |\pred(x)| \cdot |\pred(y)|
        \right).
    \]

    It follows that the time complexity is in $O\left(\left(\sum_{x \in T_0} |\pred(x)|\right)^2 \right) = O(|T|^2)$, since each non-root vertex of $T$ is the predecessor of exactly one vertex.
\end{proof}

\subsection{Decomposing zero-dimensional homology over rooted tree quiver}
For clarity, we abstract a subroutine from \cref{Algorithm 2}, given as \cref{Algorithm 2-aux}.
The following basic definitions are required \cref{Algorithm 2,Algorithm 2-aux}.

Let $G = (V,E)$ be a finite, simple, undirected graph, so that $V$ is a finite set of \emph{vertices}, and $E$ is a set of subsets of $V$, each one of cardinality $2$, called \emph{edges}.
Let $Q$ be a rooted tree quiver.
A \emph{$Q$-filtration} $h$ of $G$ consists of a pair of functions $(h_V : V \to Q_0, h_E : E \to Q_0)$, such that $h_V(x), h_V(y) \leq_{Q} h_E((x,y))$ for every $\{x,y\} \in E$.
We get a functor $(G,h) : Q \to \top$ by mapping $x \in Q_0$ to the (geometric realization of the) subgraph of $G$ spanned by vertices and edges whose $h$-value $y \in Q_0$ satisfies $y \leq_Q x$.

If $Q$ is a rooted tree quiver and $(G,h)$ is a $Q$-filtration, the \emph{level} of a vertex $v$ (respectively edge $e$) of $G$ is the level of $h_V(v) \in Q_0$ (respectively $h_E(v) \in Q_0$).

In \cref{Algorithm 2-aux}, we make use of the union-find data structure (also known as a disjoint set data structure), see, e.g., \cite[Chapter~2]{tarjan} for details.

\begin{proposition}
    \label{proposition:algo-2-proof}
    Let $Q$ be a rooted tree quiver.
    Given $(G,h)$ a $Q$-filtered graph, \cref{Algorithm 2} runs in $O(|G|^2)$ time, and returns $(S, f_S : S \to Q)$ a quiver over $Q$, such that $\kbb_S \cong H_0(h) \in \rep(Q)$, and such that $\kbb_{C}$ is indecomposable for every connected component $C \subseteq S$.
\end{proposition}
\begin{proof}
    Thanks to \cref{proposition:algo-1-proof}, it is clear that \cref{Algorithm 2} is correct as long as \cref{Algorithm 2-aux} is, and that its time complexity is in $O(|G|^2)$ as long as the time complexity of \cref{Algorithm 2-aux} is in $O(|G|^2)$.
    So let us only worry about \cref{Algorithm 2-aux}.

    Striclty speaking, in order to get a $O(|G|^2)$ complexity, we need to have $|Q| \in O(|G|)$, which can be guaranteed by restricting the codomain of $h$ to $Q' = \im(h) \cup \{\sigma\} \subseteq Q$, which is again a rooted tree poset.
    In practice, this can be achieved by skipping all the levels in line 5 of \cref{Algorithm 2-aux} that contain no vertex or edge.

    Then, the fact that the time complexity of \cref{Algorithm 2-aux} is in $O(|G|^2)$ is clear, since it iterates a constant number of times over each vertex and edge of $G$, and each time it performs $O(1)$ operations, or operations with the union find, all of which are in $O(|G|)$.

    We conclude with the correctness proof for \cref{Algorithm 2-aux}.
    By \cref{lemma:linearization-is-free}, it is enough to prove that $\fib_T \cong \pi_0(G,h) : Q \to \set$.
    This is clear, since all the algorithm does is to keep track of the connected components in the filtration $(G,h)$, and of where this connected components map under the arrows of $Q$.
\end{proof}

\begin{proof}[Proof of \cref{thm::algorithm}]
    This is a direct consequence of \cref{proposition:algo-1-proof,proposition:algo-2-proof}.
\end{proof}

\section{Building representations and filtrations over rooted tree posets}
\label{section:TDA-applications}

We now describe two ways in which to build rooted tree modules over rooted tree quivers that are relevant to topological data analysis.

\subsection{Restriction}
A \emph{restriction} of a poset $(R,\leq_R)$ is a poset $(P,\leq_P)$ such that $P \subseteq R$, and such that $p \leq_P q$ implies $p \leq_R q$ for every $p,q\in P$.
A restriction $P$ of $R$, induces a forgetful restriction functor $\top^R \to \top^P$.
So in order to study the zero-dimensional persistent homology of a filtration by $R$, one can consider its restriction to various $P \subseteq R$, with $P$ a rooted tree poset, as is standard in persistence; see, e.g., \cite[Section~4.3]{botnan-lesnick-2} and~\cite{amiot-brustle-hanson}.

In \cref{figure:restrictions-of-grid} we give simple examples of restriction of a two-dimensional grid poset, a common type of poset in multiparameter persistence \cite{escolar-hiraoka,lesnick-wright}.
This idea gives rise to a generalization of the \emph{fibered barcode} \cite{lesnick-wright-rivet} and of the foliation method \cite{cagliari-di-fabio-ferri}.
In \cref{experiments}, we apply this idea to two-parameter persistence.

\begin{figure}
    \includegraphics[width=0.8\linewidth]{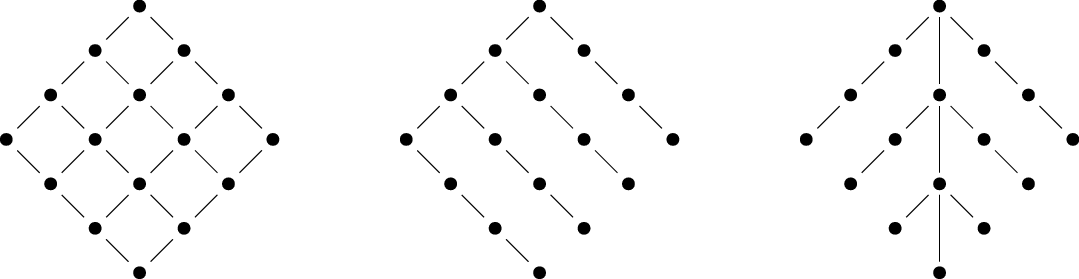}
    \caption{%
        \emph{Left.} The Hasse diagram of a two-dimensional grid poset $R$ (i.e., a product of two linear orders).
        \emph{Center and right.} Restrictions of $R$ that are rooted tree posets.
    }
    \label{figure:restrictions-of-grid}
\end{figure}

\subsection{An invariant of morphisms between merge trees}
\label{example:pdfs}
To every morphisms between merge trees, we associate a representation of the codomain, which can be effectively decomposed into indecomposables using the results in this paper.

\begin{figure}
    \includegraphics[width=\linewidth]{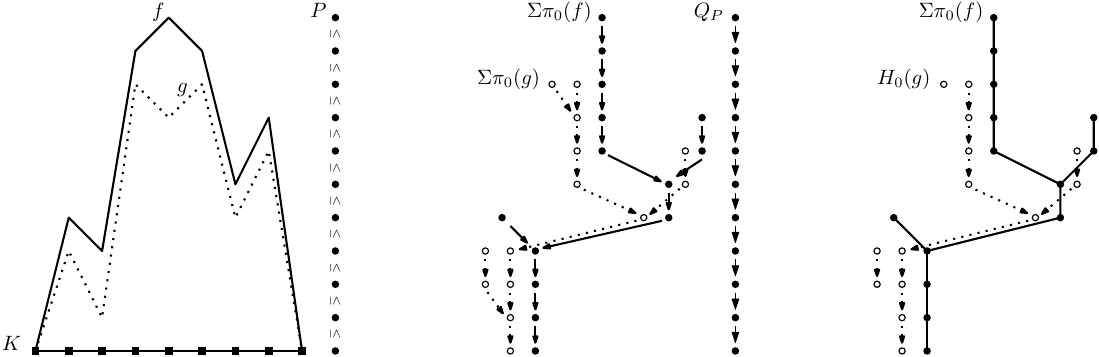}
    \caption{%
        \emph{Left.} Two filtrations $f,g : K \to P = A_n$ of a simplicial complex $K$ (with vertices depicted as squares) by a linear ordered set $P$ such that $f \leq g$.
        \emph{Center.} The connected components of the filtrations $f$ and $g$ as quivers over $Q_P$.
        \emph{Right.} The decomposition of the homology $H_0(g)$ as a representation of the rooted tree quiver $\Sigma\, \pi_0(f)$ given by the connected components of $f$.
        See \cref{example:pdfs} for details.}
    \label{figure:pdfs}
\end{figure}

\medskip \noindent \textbf{Merge trees from filtrations.}
Let $A_n$ be the equioriented quiver of type A, and denote also by $A_n = \{1 < \cdots < n\}$ the finite linear order.
Given a filtration $K_1 \subseteq \cdots \subseteq K_n$ (or a functor $K : A_n \to \top$), its connected components $\pi_0(K) : A_n \to \set$ can be visualized as a tree, known as a \emph{merge tree}, which is a fundamental object in persistence theory and in hierarchical clustering; see, e.g., \cite{curry,curry-2,rolle-scoccola} and references therein.
Related persistence-based constructions for rooted neuronal trees appear in \cite{kanari-et-al,li-et-al,beers-harrington-goriely,beers-et-al}.
There are several ways of defining merge trees; the following is a natural definition using the language of this paper.

\begin{definition}
    A \emph{merge tree} is a rooted tree quiver $(T, f_T : T \to A_n)$.
\end{definition}

We say that a functor $X : A_n \to \top$ is \emph{connected} if $X(n) \in \top$ is path-connected.
If $X : A_n \to \top$ is connected, the \emph{merge tree of} $X$ is $\Sigma\, \pi_0(X)$, obtained by applying the equivalence of categories
$\Sigma\, : \set^{A_n} \to \rtrees_{/A_n}$ (\cref{proposition:set-representation-is-tree-over}) to 
the connected components $\pi_0(X) : A_n \to \set$ of $X$.
See \cref{figure:pdfs} for examples.

\medskip \noindent \textbf{Morphisms of merge trees.}
A natural transformation $X \to Y$ of functors $X,Y : A_n \to \top$ induces a natural transformation $\pi_0(X) \to \pi_0(Y)$, and thus a morphism of $\Sigma\, \pi_0(X) \to \Sigma\, \pi_0(Y)$ of merge trees (note that connectedness of $X$ and $Y$ is not strong assumption since one can work component-wise, in the sense that $X$ and $Y$ decompose as a disjoint union of connected functors, and each component of $X$ maps to a single component of $Y$).
This situation occurs naturally, as shown next.

\begin{example}
    \label{example:natural-transformations-tda}
\begin{enumerate}
    \item Let $Z : A_n \times \{1,2\} \to \top$ be a filtration over a commutative ladder, in the sense of \cite{escolar-hiraoka}.
    Then, by letting $X \coloneqq Z|_{A_n \times \{1\}}$ and $Y \coloneqq Z|_{A_n \times \{2\}}$ we get a natural transformation $X \to Y$.
    \item Let $f,g : K \to A_n$ be two filtrations of the same simplicial complex $K$ (i.e., functions mapping simplices of $K$ to vertices of $A_n$, and respecting the face relation), and assume that $f \leq g$, such as in \cite{MR4934388}.
    Then, by letting $X \coloneqq \Scal g$ and $Y \coloneqq \Scal f$, we obtain a natural transformation $X \to Y$, since we have an inclusion of sublevel-set filtrations $(\Scal g)(r) \subseteq (\Scal f)(r)$ for every $r \in A_n$ (recall that $(\Scal f)(r) = \{\eta \in K : f(\eta) \leq r\}$).
\end{enumerate}
\end{example}

\medskip \noindent \textbf{The representation associated to a morphism of merge trees.}
The data of a morphism of merge trees $S \to T$ in $\rtree_{/A_n}$ is simply that of a rooted tree morphism $S \to T$.
Thus, we can consider the representation $\kbb_S \in \rep(T)$.

In the setup where we start with a natural transformation $X \to Y$ between connected functors $X,Y : A_n \to \top$, we get a representation $\kbb_{\Sigma\, \pi_0(X)} \in \rep(\Sigma\, \pi_0(Y))$, which can be effectively decomposed into indecomposables thanks to \cref{thm::algorithm}.
See \cref{figure:pdfs} for an example using the construction of \cref{example:natural-transformations-tda}(2).

\section{Computational Examples}
\label{experiments}

In this section, we study zero-dimensional degree--Rips persistence for three planar point clouds.
For each point cloud, the two computations use the same vector spaces and different orders on the parameter values.
First, we order the parameter values using the rooted tree illustrated on the right of \cref{figure:restrictions-of-grid}, and apply the algorithms of \cref{sect: Algorithms}.
Second, we use the product order on the full grid, and decompose the resulting module with AIDA \cite{dey-jendrysiak-kerber}.
Since the two orders retain different maps, their indecomposable decompositions need not agree.
The implementation and experiment scripts are available in the accompanying software repository \cite{scoccola-bindua-rooted-tree-code}.

\subsection{Degree--Rips modules and indexing posets}
Given a finite point cloud $S$ and a distance parameter $r$, let $G_r$ be its Vietoris--Rips graph.
For a nonnegative integer $k$, let $G_{r,k}$ be the subgraph of $G_r$ induced by the vertices having degree at least $k$ in $G_r$.
The degree--Rips bifiltration is the collection of graphs $G_{r,k}$, ordered by increasing $r$ and decreasing $k$.
Applying $H_0$ gives a persistence module on a finite rectangular grid.

For each range in \cref{table:example-configurations}, we sample 101 uniformly spaced values on each parameter axis.
At normalized degree $\delta$, the raw-degree threshold is the nearest integer to $\delta |S|$.
For the rooted-tree computation, we use all $101\times101$ parameter values, with a diagonal spine, horizontal branches below it, and vertical branches above it.
For the product-order computation, several normalized-degree values may give the same integer $k$.
We keep one row for each value of $k$, since repeated rows give the same graphs and identity maps between them.
This gives product grids with 33, 25, and 70 degree values.
Both computations use the full parameter ranges in the table.

\subsection{Data and configurations}
The Blobs example was generated with random seed 0 and consists of three disk-shaped clusters of 100 points and 100 uniformly sampled background points.
Aggregation is the 788-point, seven-class \texttt{sipu/aggregation} data set distributed in Clustering Benchmarks v1.1.0.
The collection attributes this data set to \cite{gionis-mannila-tsaparas}.
The clusterable example is the 2309-point benchmark distributed with hdbscan and used in \cite{mcinnes-healy}.
The class labels provide context, but are not used in the computations.

\begin{table}[!htbp]
\centering
\begin{tabular}{@{}lrrrr@{}}
\toprule
\textbf{Data} & \textbf{Points} & \textbf{Distance range} & \textbf{Degree range} & \textbf{Full grid} \\
\midrule
Blobs & $400$ & $[0,20]$ & $[0,0.08]$ & $101\times33$ \\
Aggregation & $788$ & $[0,1.5]$ & $[0,0.03]$ & $101\times25$ \\
clusterable & $2309$ & $[0,0.05]$ & $[0,0.03]$ & $101\times70$ \\
\bottomrule
\end{tabular}
\caption{Parameter ranges used in the examples.
The degree range is normalized by the number of points.
The rooted-tree computation uses all $101\times101$ parameter values.
The full-grid column gives the size of the product grid after repeated integer degree values are removed.}
\label{table:example-configurations}
\end{table}

\subsection{Rooted-tree decompositions}

We order the summands by total dimension, the sum of their dimensions over all parameter values.

\begin{figure}[!tbp]
    \centering
    \makebox[\textwidth][c]{\includegraphics[width=1.3\textwidth]{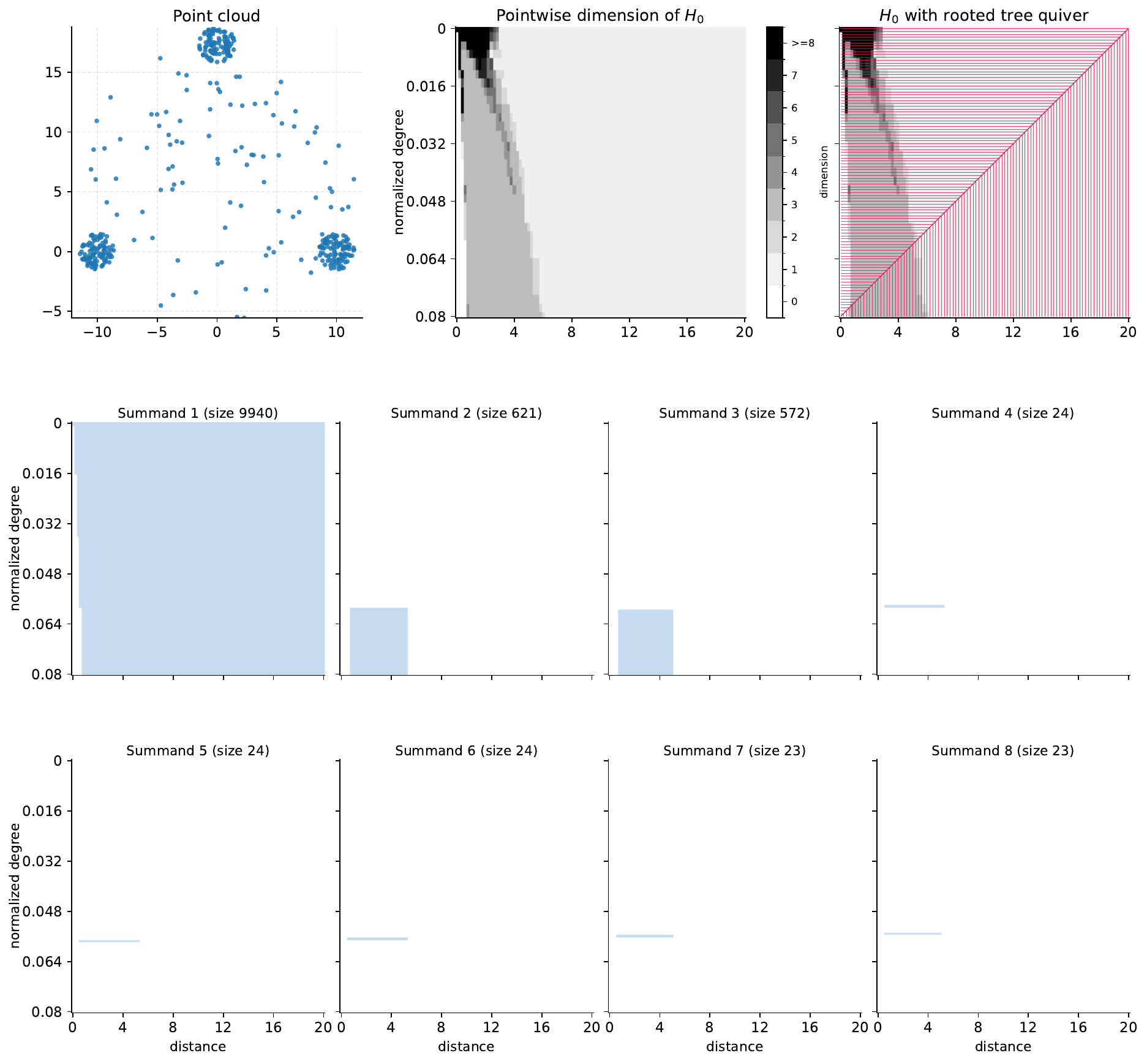}}
    \caption{Rooted-tree decomposition for Blobs.
    The first row shows the point cloud, the pointwise dimension of $H_0$, and the rooted-tree order overlaid on the same dimension plot.
    The remaining two rows show the eight largest indecomposable summands; the first three are much larger than the rest.
    All dimensions of at least $8$ have the same shade of gray, and all dimensions of at least $3$ have the same shade of blue.}
    \label{fig:decomposition}
\end{figure}

\begin{figure}[!tbp]
    \centering
    \makebox[\textwidth][c]{\includegraphics[width=1.3\textwidth]{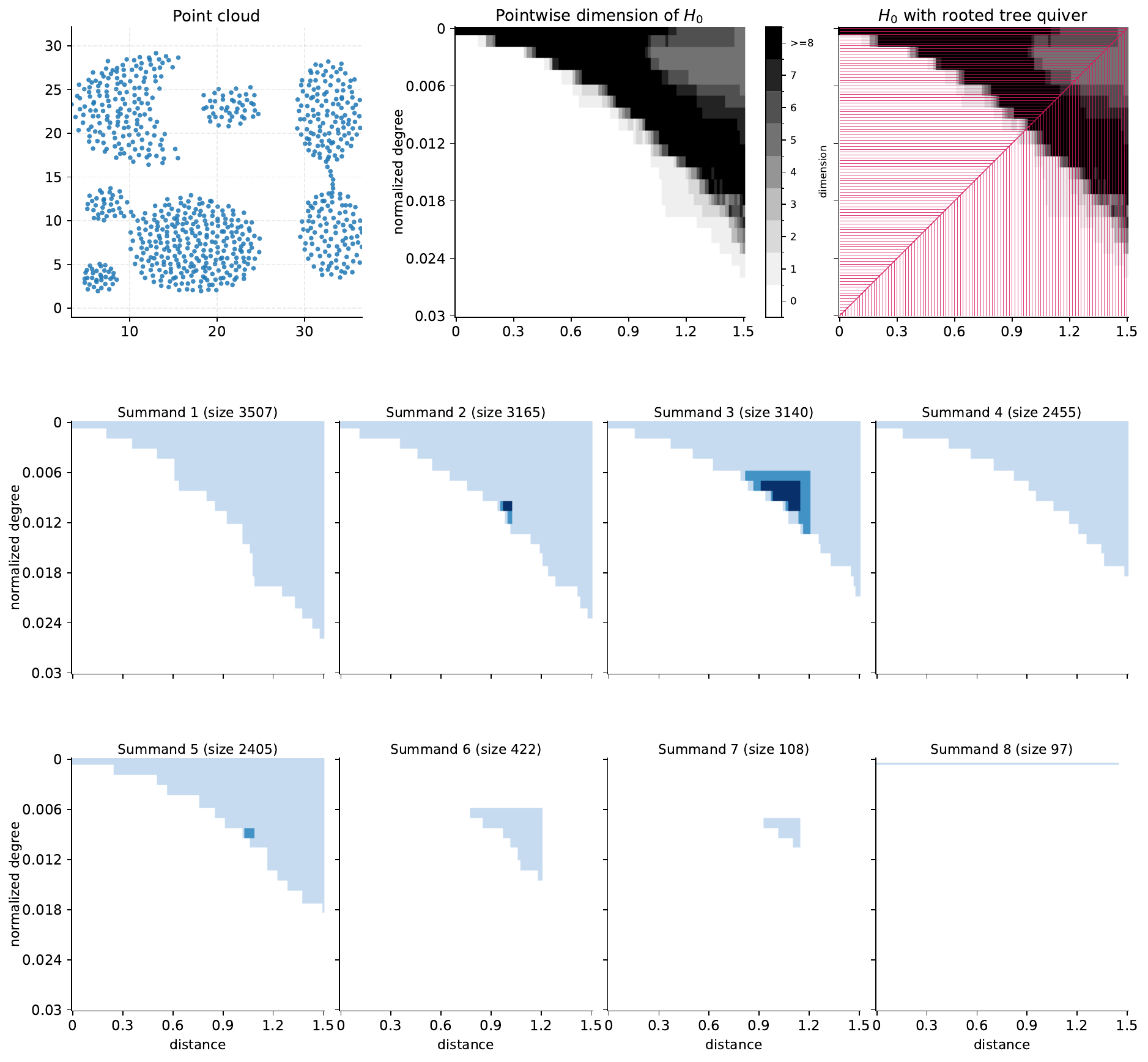}}
    \caption{Rooted-tree decomposition for Aggregation.
    The first row has the same organization as \cref{fig:decomposition}, and the remaining rows show the eight largest summands.
    Summands 1 through 7 appear in several degree rows, while summand 8 appears in just one.
    The color conventions are those of \cref{fig:decomposition}.}
    \label{fig:decomposition-aggregation}
\end{figure}

\begin{figure}[!tbp]
    \centering
    \makebox[\textwidth][c]{\includegraphics[width=1.3\textwidth]{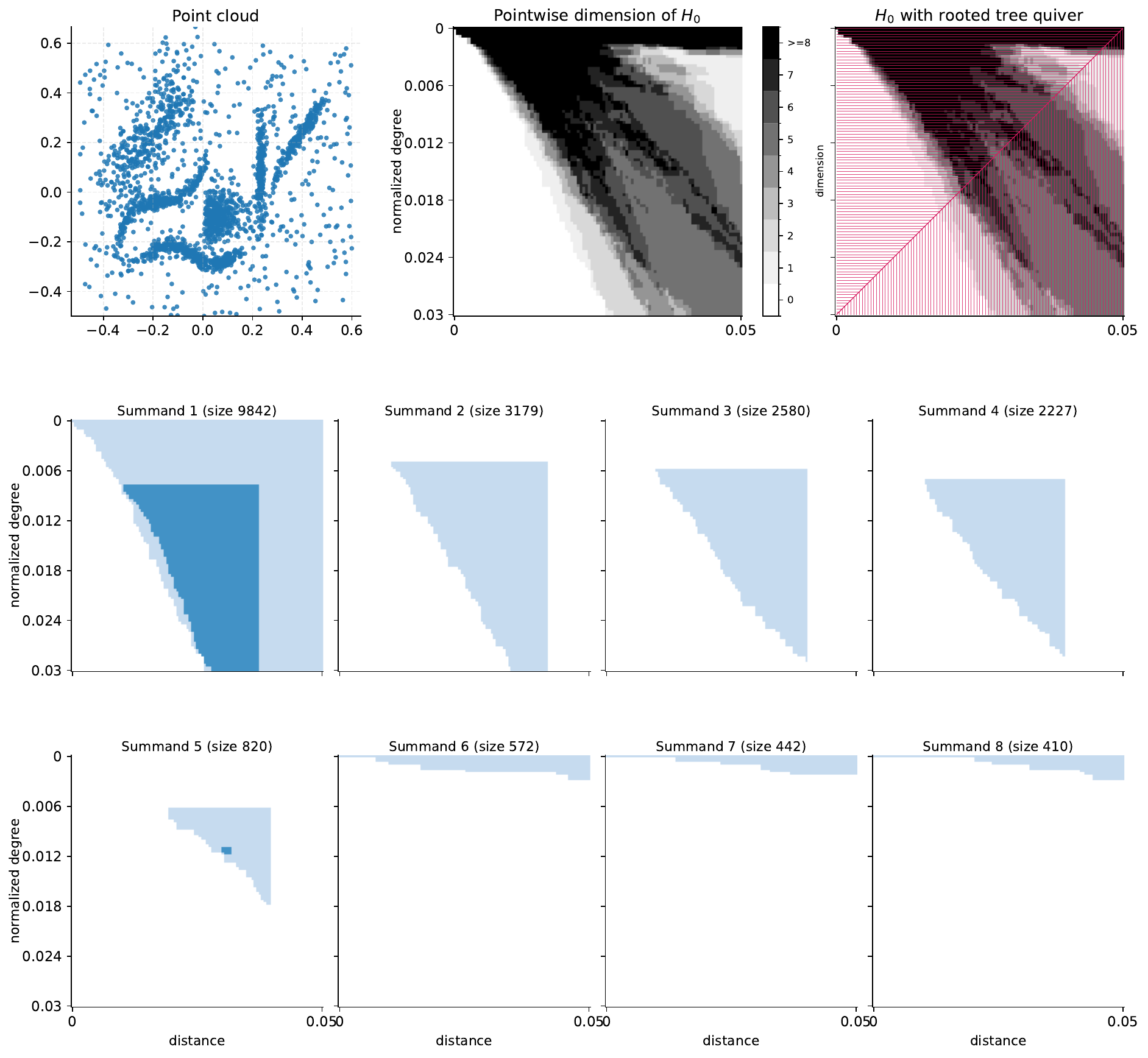}}
    \caption{Rooted-tree decomposition for clusterable.
    The first row has the same organization as \cref{fig:decomposition}, and the remaining rows show the eight largest summands.
    The first five summands extend far in both parameter directions, while the next three lie near the low-degree edge of the range.
    The color conventions are those of \cref{fig:decomposition}.}
    \label{fig:decomposition-clusterable}
\end{figure}

\FloatBarrier

\subsection{Full-grid decompositions and comparison}

For each product grid, we compute a minimal presentation of the $H_0$ module by generators and relations, and then use AIDA in brute-force mode.
The eight largest summands are shown in \cref{fig:direct-decomposition-synthetic-full,fig:direct-decomposition-aggregation-full,fig:direct-decomposition-clusterable-full}.

\Cref{table:example-results} records the exact runs used to produce all six final figures.
For all three data sets, the two computations give exactly the same dimension at every parameter value present in both grids.
For each product-order computation, the minimal presentation and the sum of the AIDA summands also agree exactly with the dimensions of the original $H_0$ module.

\begin{table}[!htbp]
\centering
{\renewcommand{\arraystretch}{1.08}
\setlength{\tabcolsep}{4pt}
\begin{tabular}{@{}llrrrr@{}}
\toprule
\textbf{Data} & \textbf{Order} & \textbf{Grid} & \textbf{Presentation} & \textbf{Summands} & \textbf{Time (s)} \\
\midrule
Blobs & rooted tree & $101\times101$ & -- & $1700$ & $1.900$ \\
Blobs & product & $101\times33$ & $570/638/244$ & $373$ & $6.106$ \\
Aggregation & rooted tree & $101\times101$ & -- & $5524$ & $0.996$ \\
Aggregation & product & $101\times25$ & $1679/2087/62$ & $926$ & $2.853$ \\
clusterable & rooted tree & $101\times101$ & -- & $6569$ & $3.734$ \\
clusterable & product & $101\times70$ & $5066/6281/171$ & $2586$ & $46.453$ \\
\bottomrule
\end{tabular}
}
\caption{Final example computations.
For the product-order computations, the presentation column gives the numbers of generators and relations, followed by the largest number of relations at one grid point.
The times are those of the runs used to produce the figures; figure generation itself is excluded.
All six runs were executed on an Apple M1 Pro machine with 16 GB RAM.
The rooted-tree times include construction of the degree--Rips filtration, restriction to the rooted tree, computation of $H_0$, and decomposition.
The product-order times include construction of the product-grid module and its presentation, AIDA, and processing its output.}
\label{table:example-results}
\end{table}

\begin{figure}[!htbp]
    \centering
    \includegraphics[width=\textwidth]{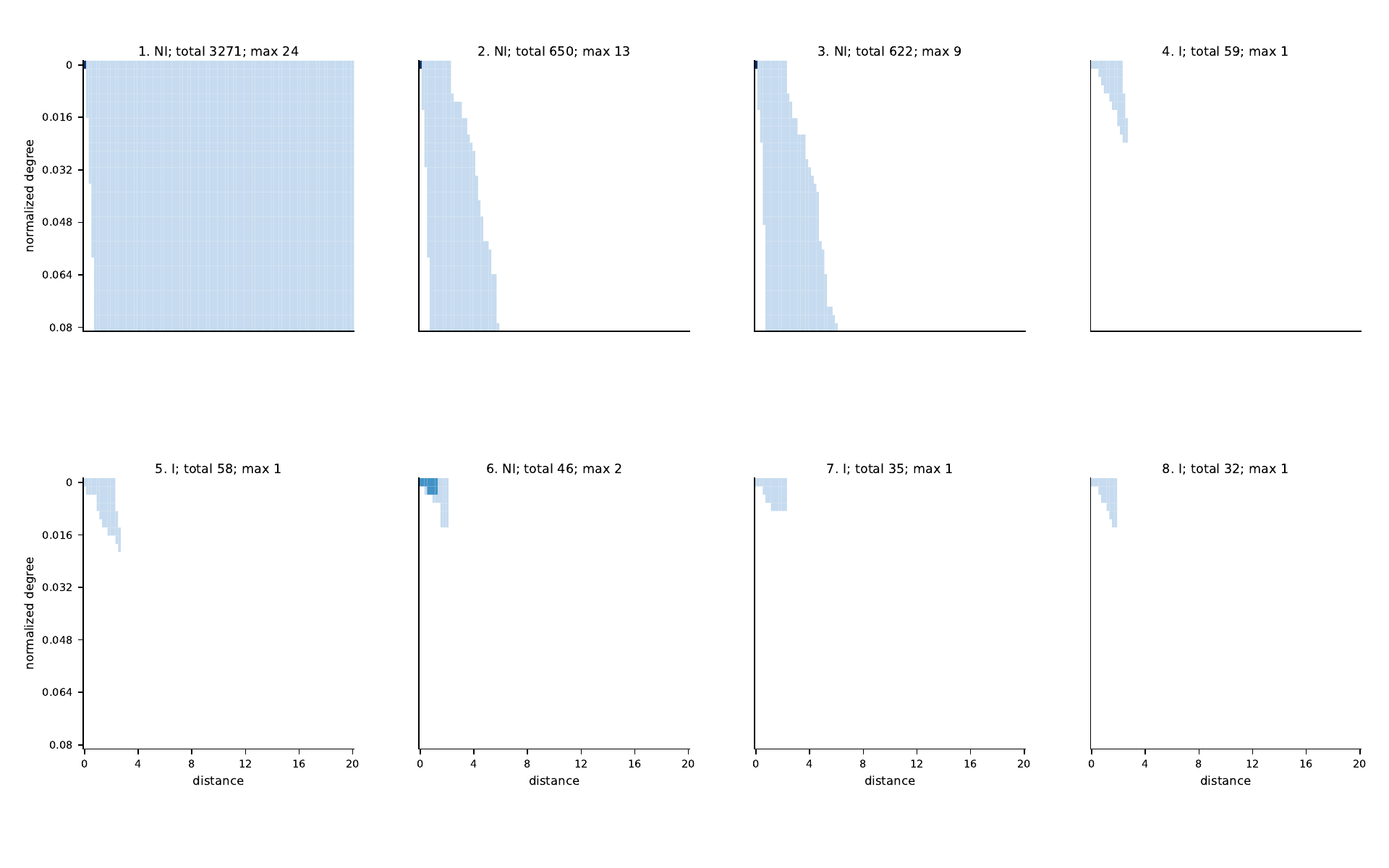}
    \caption{Full-grid decomposition for Blobs.
    The first three summands are non-spread modules.
    Their dimensions at a single grid point reach $24$, $13$, and $9$, respectively.
    Panel labels use I for spread and NI for non-spread.}
    \label{fig:direct-decomposition-synthetic-full}
\end{figure}

\begin{figure}[!htbp]
    \centering
    \includegraphics[width=\textwidth]{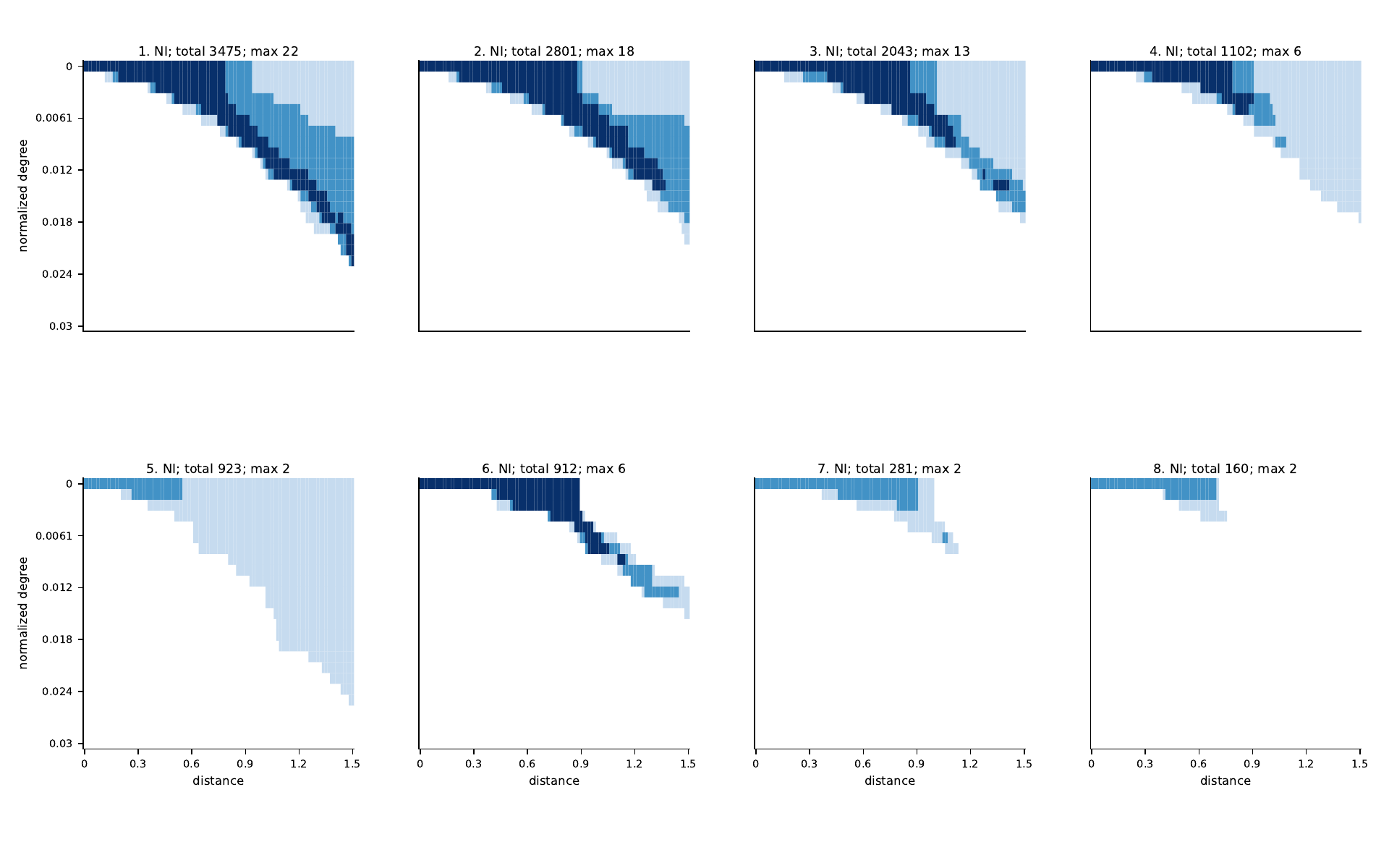}
    \caption{Full-grid decomposition for Aggregation.
    All eight displayed summands are non-spread modules, and each of the first six has dimension greater than one at some grid point.
    Panel labels use the conventions of \cref{fig:direct-decomposition-synthetic-full}.}
    \label{fig:direct-decomposition-aggregation-full}
\end{figure}

\begin{figure}[!htbp]
    \centering
    \includegraphics[width=\textwidth]{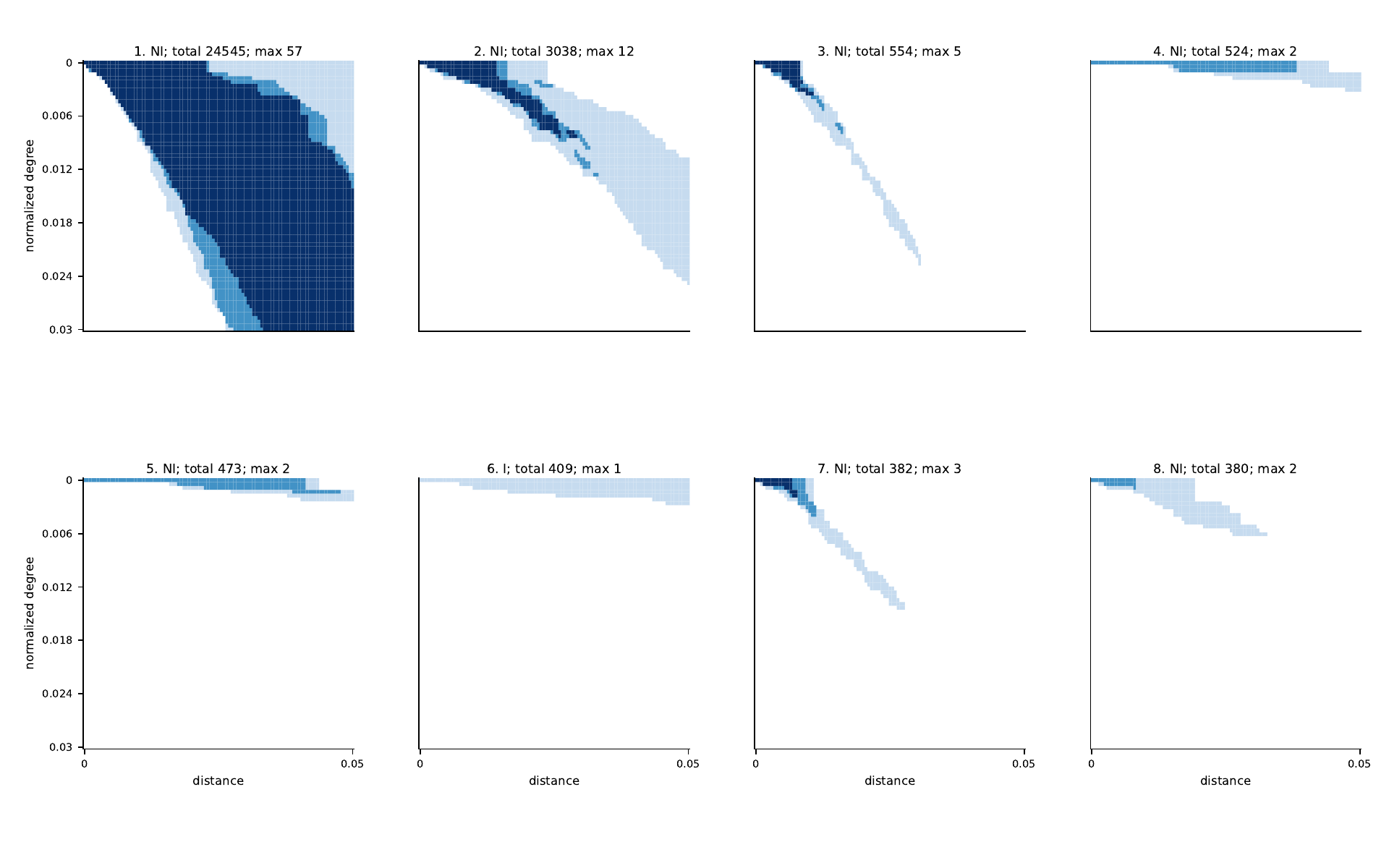}
    \caption{Full-grid decomposition for clusterable.
    Seven of the eight displayed summands are non-spread modules; summand 6 is a spread.
    Panel labels use the conventions of \cref{fig:direct-decomposition-synthetic-full}.}
    \label{fig:direct-decomposition-clusterable-full}
\end{figure}

\FloatBarrier

\subsection{Running time and memory use}
We use a separate synthetic experiment to study how time and memory depend on the number $n$ of points and the size of the poset $P$.
To vary the size of $P$ without changing the shape of the point cloud, we scale the whole point cloud by a constant.
The timings were measured on a machine with an Apple M$1$ Pro chip and $16$ GB RAM.
Each configuration was run three times in a fresh Python process after a short warmup.
The tables give the mean and sample standard deviation of the three runs.
\cref{table:running_time} records the time required to construct the rooted-tree filtration, compute $H_0$, and decompose the resulting module.
\cref{table:memory} records increases in the memory held by the process during the same runs.
For each phase, we report the largest increase over the memory in use at the start of that phase.
The total peak is the largest increase over the memory in use before the computation begins.
\cref{table:bifiltration_sizes} compares $\sum_{i,j}(|V_{i,j}|+|E_{i,j}|)$ before and after each graph is replaced by a forest with the same connected components.
The total for the full degree--Rips bifiltration was computed without storing all of its graphs.
For comparison, \cref{table:running_time,table:memory} also include direct AIDA decompositions over the same product grids whenever the computation finished.
These runs include construction of the product-grid module and its minimal presentation, followed by AIDA.
They exclude figure generation and the computation of the dimensions of the individual summands.

For grid widths $10$ and $10^2$, constructing the filtration takes most of the time in the rows with the most points, since dense clusters produce many edges.
For grid width $10^3$, all three phases take a substantial amount of time, and constructing the filtration is still the longest phase in the largest completed row.

\begin{table}[!htbp]
\centering
{\renewcommand{\arraystretch}{1.08}
\tiny
\setlength{\tabcolsep}{1.5pt}
\begin{tabular}{@{}ccrrrrr@{}}
\toprule
\multicolumn{2}{c}{\textbf{Configuration}} & \multicolumn{5}{c}{\textbf{Time (s)}} \\
\cmidrule(lr){1-2}\cmidrule(l){3-7}
\textbf{Points} & \textbf{Poset size (grid width)} & \textbf{Filtration} & \textbf{$H_0$} & \textbf{Tree decomp.} & \textbf{Tree total} & \textbf{AIDA total} \\
\midrule
\multirow{3}{*}{$10^2$} & $10^2$ ($10$) & $0.012 \pm 0.001$ & $0.002 \pm 0.000$ & $0.001 \pm 0.000$ & $0.077 \pm 0.001$ & $0.36$ \\
& $10^4$ ($10^2$) & $0.48 \pm 0.01$ & $0.057 \pm 0.000$ & $0.033 \pm 0.002$ & $0.67 \pm 0.01$ & $1.24$ \\
& $10^6$ ($10^3$) & $20.9 \pm 0.4$ & $8.19 \pm 1.04$ & $11.4 \pm 1.5$ & $43.4 \pm 2.9$ & $98.3$ \\
\midrule
\multirow{3}{*}{$10^3$} & $10^2$ ($10$) & $0.35 \pm 0.00$ & $0.023 \pm 0.000$ & $0.002 \pm 0.000$ & $0.44 \pm 0.00$ & $6.29$ \\
& $10^4$ ($10^2$) & $3.73 \pm 0.02$ & $0.71 \pm 0.00$ & $0.025 \pm 0.000$ & $4.68 \pm 0.04$ & $24.3$ \\
& $10^6$ ($10^3$) & $121.2 \pm 4.0$ & $64.4 \pm 2.6$ & $37.5 \pm 2.8$ & $254.0 \pm 9.0$ & $310.1^{*}$ \\
\midrule
\multirow{2}{*}{$10^4$} & $10^2$ ($10$) & $47.7 \pm 0.4$ & $0.38 \pm 0.06$ & $6.79 \pm 0.04$ & $55.0 \pm 0.4$ & $>450$ \\
& $10^4$ ($10^2$) & $372.6 \pm 5.2$ & $16.2 \pm 0.5$ & $0.59 \pm 0.02$ & $398.7 \pm 5.7$ & -- \\
\bottomrule
\end{tabular}
}
\caption{Phase timings in seconds for the synthetic three-cluster experiment.
The rooted-tree entries are mean $\pm$ sample standard deviation over three runs in fresh Python processes.
The AIDA total is the time for one direct decomposition in a fresh process.
In the starred row, repeated degree rows were removed, reducing their number from $1001$ to $101$; the result is therefore a lower bound for the full product grid.
The $>450$ entry means that the presentation was constructed, but AIDA was stopped after more than seven minutes; -- indicates a computation that was not run.}
\label{table:running_time}
\end{table}

\begin{table}[!htbp]
\centering
{\renewcommand{\arraystretch}{1.08}
\tiny
\setlength{\tabcolsep}{1.5pt}
\begin{tabular}{@{}ccrrrrr@{}}
\toprule
\multicolumn{2}{c}{\textbf{Configuration}} & \multicolumn{5}{c}{\textbf{Memory (MiB)}} \\
\cmidrule(lr){1-2}\cmidrule(l){3-7}
\textbf{Points} & \textbf{Poset size (grid width)} & \textbf{Filtration} & \textbf{$H_0$} & \textbf{Tree decomp.} & \textbf{Tree peak} & \textbf{AIDA peak} \\
\midrule
\multirow{3}{*}{$10^2$} & $10^2$ ($10$) & $1.1 \pm 0.1$ & $0.0 \pm 0.0$ & $0.0 \pm 0.0$ & $1.1 \pm 0.1$ & $1.1$ \\
& $10^4$ ($10^2$) & $18.6 \pm 5.0$ & $11.9 \pm 0.9$ & $4.7 \pm 2.1$ & $35.3 \pm 5.4$ & $79.7$ \\
& $10^6$ ($10^3$) & $727.5 \pm 7.1$ & $1146.4 \pm 93.4$ & $405.6 \pm 133.3$ & $2279.4 \pm 226.8$ & $2750.0$ \\
\midrule
\multirow{3}{*}{$10^3$} & $10^2$ ($10$) & $37.8 \pm 0.1$ & $0.7 \pm 0.0$ & $0.0 \pm 0.0$ & $38.5 \pm 0.1$ & $18.8$ \\
& $10^4$ ($10^2$) & $368.8 \pm 6.9$ & $14.8 \pm 7.1$ & $0.0 \pm 0.0$ & $383.6 \pm 14.0$ & $408.8$ \\
& $10^6$ ($10^3$) & $2444.2 \pm 401.1$ & $185.5 \pm 172.2$ & $27.5 \pm 47.6$ & $2891.3 \pm 164.3$ & $1976.6^{*}$ \\
\midrule
\multirow{2}{*}{$10^4$} & $10^2$ ($10$) & $3460.4 \pm 447.0$ & $0.0 \pm 0.0$ & $0.0 \pm 0.0$ & $3460.4 \pm 447.0$ & $\ge 1418.0$ \\
& $10^4$ ($10^2$) & $2968.3 \pm 201.3$ & $92.1 \pm 159.5$ & $0.0 \pm 0.0$ & $3361.0 \pm 190.9$ & -- \\
\bottomrule
\end{tabular}
}
\caption{Additional memory used, in MiB, for the same runs as \cref{table:running_time}.
The rooted-tree entries are mean $\pm$ sample standard deviation over three runs in fresh Python processes.
The AIDA peak is the largest observed value in one direct run, including the memory used by the AIDA subprocess.
The starred row is the same lower-bound run as in \cref{table:running_time}; $\ge$ marks the incomplete direct run.}
\label{table:memory}
\end{table}

\begin{table}[!htbp]
\centering
{\renewcommand{\arraystretch}{1.08}
\scriptsize
\setlength{\tabcolsep}{3pt}
\begin{tabular}{@{}ccrrr@{}}
\toprule
\multicolumn{2}{c}{\textbf{Configuration}} & \multicolumn{3}{c}{\textbf{Total size}} \\
\cmidrule(lr){1-2}\cmidrule(l){3-5}
\textbf{Points} & \textbf{Poset size (grid width)} & \textbf{Full degree-Rips} & \textbf{Forest-reduced} & \textbf{Ratio} \\
\midrule
\multirow{3}{*}{$10^2$} & $10^2$ ($10$) & $1.54 \times 10^5$ & $6.81 \times 10^3$ & $22.6$ \\
& $10^4$ ($10^2$) & $1.34 \times 10^7$ & $1.45 \times 10^5$ & $92.6$ \\
& $10^6$ ($10^3$) & $1.32 \times 10^9$ & $1.63 \times 10^6$ & $810.7$ \\
\midrule
\multirow{3}{*}{$10^3$} & $10^2$ ($10$) & $1.48 \times 10^7$ & $8.70 \times 10^4$ & $170.6$ \\
& $10^4$ ($10^2$) & $1.29 \times 10^9$ & $2.74 \times 10^6$ & $469.6$ \\
& $10^6$ ($10^3$) & $1.27 \times 10^{11}$ & $1.09 \times 10^8$ & $1159.9$ \\
\midrule
\multirow{2}{*}{$10^4$} & $10^2$ ($10$) & $1.48 \times 10^9$ & $9.76 \times 10^5$ & $1512.1$ \\
& $10^4$ ($10^2$) & $1.28 \times 10^{11}$ & $4.67 \times 10^7$ & $2738.7$ \\
\bottomrule
\end{tabular}
}
\caption{Total number of vertices and edges over all parameter values in the synthetic three-cluster experiment.
Full degree--Rips refers to the original bifiltration, and forest-reduced refers to the result after each graph is replaced by a forest with the same connected components.}
\label{table:bifiltration_sizes}
\end{table}

\FloatBarrier

\subsection{Discussion}
The two decompositions answer different questions.
The rooted-tree order keeps fewer maps between parameter values, which makes a direct indecomposable decomposition possible.
The product order keeps all maps in the grid, but its largest summands are often non-spread modules with dimension greater than one at some grid points.
The rooted-tree decomposition therefore describes the chosen restriction; it is not an approximation of the product-order decomposition.

In all three examples, the dimensions agree at every parameter value present in both grids.
In the scaling experiment, the forest reduction decreases the total number of vertices and edges by factors ranging from $22.6$ to $2738.7$.
For the larger grids tested, the current implementation of the direct product-order decomposition takes more time.

\bigskip
\subsection*{Acknowledgements}
TB was supported by Bishop's University, Universit\'e de Sherbrooke and NSERC Discovery Grant RGPIN/04465-2019.
RB was supported by a Mitacs Globalink research award.
A portion of this work was completed while RB was at Universit\'e de Sherbrooke.
LS was supported by a Centre de Recherches Mathématiques et Institut des Sciences Mathématiques fellowship.

\bibliographystyle{plain}
\bibliography{references}

\end{document}